\documentclass[a4paper,11pt]{article}
\usepackage{pdfsync}

 \usepackage[plainpages=false]{hyperref}
 \usepackage{amsfonts,amsmath,amssymb,amsthm,enumerate}
 \usepackage{latexsym,lscape,rawfonts,mathrsfs,mathtools}
\usepackage{enumitem}
\usepackage{relsize}
\makeatletter
\newcommand*\owedge{\mathpalette\@owedge\relax}
\newcommand*\@owedge[1]{%
  \mathbin{%
    \ooalign{%
      $#1\m@th\bigcirc$\cr
      \hidewidth$#1\m@th\wedge$\hidewidth\cr
    }%
  }%
}
\makeatother

 \usepackage[all]{xy}
 \usepackage{makeidx}
 \usepackage{graphicx,psfrag}
 \usepackage{pstool}
 \usepackage{tikz-cd}
 \usepackage{array,tabularx,tabu}

 \usepackage{setspace}

 \usepackage[titletoc,title]{appendix}

\usepackage{txfonts}

 \newcommand{\C}{\ensuremath{\mathbb{C}}}

 \newcommand{\R}{\ensuremath{\mathbb{R}}}

 \newcommand{\CP}{\ensuremath{\mathbb{CP}}}

 \newcommand{\ba}{\begin{align*}}
 \newcommand{\ea}{\end{align*}}
 \newcommand{\na}{\nabla}

\newcommand{\Rm}{\text{Rm}}
\newcommand{\lc}{\left(}
\newcommand{\del}{\partial}
\newcommand{\rc}{\right)}
\newcommand{\ep}{\epsilon}
\newcommand{\tl}{\left(\partial_t-\Delta \right)}
\newcommand{\ka}{K\"ahler\,}

 \makeatletter
 \def\ExtendSymbol#1#2#3#4#5{\ext@arrow 0099{\arrowfill@#1#2#3}{#4}{#5}}
 
 \makeatother

 \makeatletter
 \def\ExtendSymbol#1#2#3#4#5{\ext@arrow 0099{\arrowfill@#1#2#3}{#4}{#5}}
 
 \makeatother

\def\XXint#1#2#3{{\setbox0=\hbox{$#1{#2#3}{\int}$ }
\vcenter{\hbox{$#2#3$ }}\kern-.55\wd0}}

\numberwithin{equation}{section}

\newtheorem{thm}{Theorem}[section]
\newtheorem{cor}[thm]{Corollary}
\newtheorem{prop}[thm]{Proposition}
\newtheorem{lem}[thm]{Lemma}
\newtheorem{conj}[thm]{Conjecture}

\newtheorem{rem}[thm]{Remark}

\newtheorem{defn}[thm]{Definition}

\title{Ancient solutions to the Ricci flow with isotropic curvature conditions}
\author{Jae Ho Cho, Yu Li\footnote{Partially supported by research fund from SUNY Stony Brook.}}
\date{\today}

\begin{document}
\maketitle

\begin{abstract}
We show that every $n$-dimensional, $\kappa$-noncollapsed, noncompact, complete ancient solution to the Ricci flow with uniformly PIC for $n=4$ or $n\ge12$ has weakly PIC$_2$ and bounded curvature. Combining this with the results in \cite{BN}, we prove that any such solution is isometric to either a family of shrinking cylinders (or a quotient thereof) or the Bryant soliton.  Also, we classify all complex 2-dimensional, $\kappa$-noncollapsed, complete ancient solutions to the \ka Ricci flow with weakly PIC.
\end{abstract}

\tableofcontents

\section{Introduction}
One of the most important questions in the study of geometric flows is to understand the formation of singularity, which is likely to develop in the flow. In Ricci flow, the singularity is investigated by considering the ancient solutions introduced by Hamilton \cite{H3}. If one starts the Ricci flow from a compact manifold, Perelman has proved that any ancient solution, which is formed by the blow-up of the high curvature region, must be $\kappa$-noncollapsed by using his celebrated entropy formula \cite{P}. Therefore, it is a central issue to understand $\kappa$-noncollapsed ancient solutions to the Ricci flow. In dimension $3$, Perelman proved an important structure theorem for all ancient solutions with positive curvature, by stating that it resembles the Bryant soliton and consists of a neck region and a cap region. With this characterization, Perelman proved a canonical neighborhood theorem for high curvature part and managed to continue the Ricci flow through the singular time, which eventually leads to the resolution of Poincar\'e and geometrization conjectures for $3$-manifolds, see \cite{P} \cite{P1} \cite{P2}.

In dimension $2$, all $\kappa$-noncollapsed ancient solutions to the Ricci flow are $\R^2$ and shrinking spheres or their quotients, see \cite{P}. In fact, all ancient solutions to the Ricci flow on surface, including the collapsing case, are completely classified, see \cite{DHS12} \cite{DS06} \cite{Chu07}. In dimension $3$, all $\kappa$-noncollapsed ancient solutions on noncompact manifolds are shrinking cylinders (or their quotients) and the Bryant soliton, proved in Brendle's breakthrough \cite{B0}. For compact case, it was proved by Brendle-Daskalopoulos-Sesum \cite{BDS20} that any compact, $\kappa$-noncollapsed ancient solution is a family of shrinking spheres (or their quotients) or the Perelman's solution \cite{P2}.

In general, it is highly difficult to investigate ancient solutions without any curvature assumption. In dimension $3$, it follows from the well-known Hamilton-Ivey pinching estimate that any ancient solution must have nonnegative sectional curvature, which plays an essential role in the classification of the solution. In higher dimensional case, there is no natural curvature pinching, except that the scalar curvature is nonnegative proved by Chen \cite{C}. Therefore, it is natural to impose various positivity conditions for the curvature, which should be preserved by the Ricci flow. 

One of the natural conditions is \emph{positive isotropic curvature} (PIC for short). The PIC condition, introduced by Micallef-Moore \cite{MM}, is preserved under the Ricci flow, proved by Hamilton \cite{H1} for the 4-dimensional case and Nguyen \cite{NG} and Brendle-Schoen \cite{BS} independently in general dimension. Two other closely related curvature conditions, PIC$_1$ and PIC$_2$,  were introduced by Brendle-Schoen \cite{BS} and they play important roles in the proof of the Differentiable Sphere Theorem, see the monograph \cite{RFST}. For the precise definitions of these curvature conditions, see Definition \ref{defn:cur}.

In dimension 4, Hamilton \cite{H1} (see also \cite{CZ} \cite{CTZ12}) classified all differential structures of compact manifolds with PIC, provided that there is no essential incompressible space-form. It was proved in \cite{H1} any ancient solution developed from the blow-up process must have a nonnegative curvature operator and uniformly PIC. This curvature improvement can be regarded as a four-dimensional generalization of the Hamilton-Ivey pinching. In the higher dimension when $n \ge 12$, similar classification of all compact manifolds with PIC was obtained by Brendle \cite{B1}. It was proved by Brendle \cite{B1} that any ancient solution coming from a compact manifold with PIC must have weakly PIC$_2$ and uniformly PIC. The method used by Brendle is to construct ingeniously a family of curvature cones that pinches toward the desired curvature condition.

Therefore, it is important to investigate $\kappa$-noncollapsed ancient solutions with weakly PIC$_2$ and uniformly PIC. On the one hand, the weakly PIC$_2$ condition guarantees that many important properties, including compactness, non-Euclidean volume growth, Hamilton's differential Harnack inequality, etc., are available. On the other hand, the uniformly PIC condition has completely determined the geometry at infinity. Recently, Brendle and Naff have proved the following classification result, which is a generalization of Brendle's seminal work on 3-dimensional ancient solution \cite{B0}.

\begin{thm}[Corollary $1.6$ of Brendle-Naff~\cite{BN}] \label{thm:kappa00}
Any complete, noncompact, $\kappa$-noncollapsed ancient solution to the Ricci flow with weakly \emph{PIC}$_2$, uniformly \emph{PIC} and bounded curvature is isometric to either a family of shrinking cylinders (or a quotient thereof ) or to the Bryant soliton.
\end{thm}

The key idea of Theorem \ref{thm:kappa00}, like the 3-dimensional case, is to prove that such a solution must be rotationally symmetric, by using a deep neck improvement theorem \cite[Theorem $4.8$]{BN}, \cite[Theorem $8.5$]{B0}. By analyzing the rotationally symmetric ancient solutions with the same property, the complete classification is obtained \cite[Theorem $1.1$]{B0},\cite[Theorem $1.5$]{LZ18}.

In this paper, we slightly improve Theorem \ref{thm:kappa00} by showing that for the ancient solution, if the dimension $n=4$ or $n \ge 12$, the uniformly PIC condition implies the weakly PIC$_2$. Moreover, with the $\kappa$-noncollapsing condition, we prove that all these ancient solutions must have bounded curvature. 

The first main result of the paper is

\begin{thm} \label{thm:main1}
Let $(M^n,g(t))_{t \in (-\infty,0]}$ be a $\kappa$-noncollapsed, noncompact, complete ancient solution to the Ricci flow with uniformly \emph{PIC} for $n=4$ or $n \ge 12$. Then it is isometric to either a family of shrinking cylinders (or a quotient thereof ) or the Bryant soliton.
\end{thm}

The proof of Theorem \ref{thm:main1} consists of two parts. First, we prove that uniformly PIC implies weakly PIC$_2$. For the curvature improvement, the usual strategy is to apply the maximum principle for the curvature operator by constructing some pinching sets. However in our case, since the curvature is not assumed to be bounded on each time slice, the maximum principle is not automatically available. Here, by following the idea of Chen's proof that any scalar curvature is nonnegative for ancient solutions \cite{C}, we consider parabolic equations in the following form
\begin{align}
\tl f \ge  f^2.
\end{align}
With a localization argument, see Proposition \ref{P201}, one can prove that $f$ must be nonnegative without any curvature restriction. By this observation, we prove, in dimension $4$, that the curvature operator is nonnegative, see Lemma \ref{lem:403}. In the higher dimension, we consider a continuous family of curvature cones, constructed by Brendle \cite{B1}, and show that the curvature operator of our ancient solution lies in all those cones. Since the curvature cones pinch toward weakly PIC$_1$, the ancient solution has weakly PIC$_1$. Therefore, by \cite[Lemma 4.2]{BCW} (see also \cite[Proposition 6.2]{LN}), the ancient solution has weakly PIC$_2$. Moreover, we obtain a general theorem for curvature improvement, see Theorem \ref{thm:cones}. In the proof, the $\kappa$-noncollapsing condition is not needed. We remark that similar curvature improvement for Ricci shrinkers is obtained in \cite{Ka19}.

The second part of the proof is to show that any such ancient solution must have bounded curvature. The proof relies on a version of canonical neighborhood theorem for Ricci flows with weakly PIC$_2$ and uniformly PIC, see Theorem \ref{thm:cano}. We show that any higher curvature part of such a solution is close to a round cylinder or the Bryant soliton. By using the fact that any manifold with positive sectional curvature cannot have smaller and smaller cylinders \cite[Proposition $2.2$]{CZ}, we can derive a contradiction if the curvature is unbounded, see Proposition \ref{prop:bdd}. In the proof, the $\kappa$-noncollapsing condition is essentially used.

Our second main result is the following classification of the ancient solutions to the \ka Ricci flow on \ka surfaces.

\begin{thm}\label{thm:main2}
 Let $(M^2,g(t))_{t \in (-\infty,0]}$ be a $\kappa$-noncollapsed, complete, complex 2-dimensional ancient solution to the K\"{a}hler-Ricci flow with weakly \emph{PIC} and bounded curvature in any compact time interval. Then it is isometric-biholomorphic to one of the spaces $\C^2$, $\CP^2$, $\CP^1 \times \CP^1$, $\C^1 \times \CP^1$ equipped with standard metrics, up to scalings on each factors.
\end{thm}

In the \ka setting, a similar problem is to understand all $\kappa$-noncollapsed, ancient solutions to the \ka Ricci flow with some positive curvature conditions. It is well-known that a \ka surface (i.e. Kähler manifold with complex dimension 2) has weakly PIC if and only if it has nonnegative orthogonal bisectional curvature, which is preserved under \ka Ricci flow \cite{GZ}. In general, an ancient solution to the \ka Ricci flow with nonnegative orthogonal bisectional curvature has nonnegative bisectional curvature \cite{LN}. For the $\kappa$-solutions (see Definition \ref{def:kappa1}) to the \ka Ricci flow, there is a parallel theory to Perelman's theory on $\kappa$-solutions to the $3$-dimensional Ricci flows. Many important theorems, including compactness theorem, Hamilton's differential Harnack inequality also hold for the $\kappa$-solutions to the \ka Ricci flow \cite{N}\cite{Ca3}. Notice that all compact $\kappa$-solutions to the \ka Ricci flow are classified \cite{DZ1}. Therefore, the main issue is to prove the noncompact case.

To prove Theorem \ref{thm:main2}, we first prove that any $\kappa$-solutions to the \ka Ricci flow must be of Type I, see Lemma \ref{lem:type2}. From the classification of \ka Ricci shrinkers with nonnegative bisectional curvature \cite{N}, the asymptotic behavior of the ancient solution is modeled on standard $\C \times \CP^1$. Moreover, we obtain a curvature improvement so that the ancient solution has a nonnegative curvature operator, see Theorem \ref{lem:405}. With the nonnegative sectional curvature, we show that at spatial infinity, the geometry is also modeled on $\C \times \CP^1$. Therefore, we can prove a canonical neighborhood theorem, see Proposition \ref{prop:beh}, such that locally one can obtain a $S^2$-fibration. By a standard argument, we can patch all local fibrations to form a global fibration, see Theorem \ref{thm:kappa3}. Now the proof of Theorem \ref{thm:main2} is complete by a topological argument.

This paper is organized as follows. In section 2, we review some basic concepts of the Ricci flow and prove a localization result. In section 3 and 4, we prove the curvature improvement for dimension $n \ge 12$ and $n=4$, respectively. In section 5, we prove a canonical neighborhood theorem and Theorem \ref{thm:main1}. In section 6, we construct a fibration on the \ka Ricci flow and prove Theorem \ref{thm:main2}. In the last section, we propose some further questions.
\\
\\

\textbf{Acknowledgements:} Jae Ho Cho would like to thank his Ph.D. advisor Xiuxiong Chen for all his constant help, encouragement, and guidance. Yu Li would like to thank Professor Xiuxiong Chen and Professor Bing Wang for helpful discussions.

\section{Preliminaries}

First, we recall the following standard definitions in the Ricci flow.

\begin{defn}\label{def:singular} Let $(M^n,g(t))_{t \in (-\infty,T)}$ be a complete ancient solution to the Ricci flow $\partial_t g=-2Ric$.
\begin{enumerate}[label=(\roman*)]
\item $(M,g(t))_{t \in (-\infty,T)}$ has bounded curvature on any compact time interval if for any compact interval $I \subset (-\infty,T)$, we have $\sup_{M \times I}|\emph{Rm}(x,t)|<\infty$. In this case, $T$ is called the singular time if $T<\infty$ and $\lim_{t \to T} \sup_{M\times \{t\}} |\emph{Rm}(x,t)|=\infty$.

\item Given $\kappa>0$, $(M,g(t))_{t \in (-\infty,T)}$ is said to be $\kappa$-noncollapsed (at all scales) if for any ball $B_{g(t)}(x,r)$ satisfying $|\emph{Rm}(y,t)| \le r^{-2}$ for all $ y \in B_{g(t)}(x,r)$, we have $\emph{Vol}\,B_{g(t)}(x,r)\ge \kappa r^n$. 
\end{enumerate}
\end{defn}

Next, we recall the following curvature conditions which are used throughout the paper.

\begin{defn}\emph{(Isotropic curvature conditions) }\label{defn:cur}
\begin{enumerate}[label=(\roman*)]
\item A Riemannian manifold $(M^n,g)\,(n\ge 4)$ is said to have weakly \emph{PIC} (in other words, nonnegative isotropic curvature) if for any $p \in M$ and any orthonormal 4-frame $\{e_1,e_2,e_3,e_4\}$ of $T_pM$, we have
 $$
R_{1313}+R_{1414}+R_{2323}+R_{2424}-2R_{1234}\ge0.
$$
Note that when $n=4$, this condition is equivalent to $a_1+a_2 \ge 0$ and $c_1+c_2 \ge 0$ using the notation in \emph{\cite{H}}. 

\item A Riemannian manifold $(M^n,g)\,(n\ge 5)$ is said to have uniformly \emph{PIC} if there exists a constant $\theta>0$ such that for any $p \in M$ and any orthonormal 4-frame $\{e_1,e_2,e_3,e_4\}$ of $T_pM$ we have
$$
R_{1313}+R_{1414}+R_{2323}+R_{2424}-2R_{1234}\ge 4\theta R>0. \label{E:upic}
$$

\item A 4-dimensional Riemannian manifold $(M^4,g)$ is said to have uniformly \emph{PIC} if there exists a constant $\Lambda\ge1$ such that
$$
0<\max(a_3,b_3,c_3)\le \Lambda \min(a_1+a_2,c_1+c_2).
$$

\item A Riemannian manifold $(M^4,g)$ is said to have weakly \emph{PIC}$_1$ if for any $\lambda \in [-1,1]$, any $p \in M$ and any orthonormal 4-frame $\{e_1,e_2,e_3,e_4\}$ of $T_pM$, we have
$$
R_{1313}+\lambda^2 R_{1414}+ R_{2323}+\lambda^2 R_{2424}-2\lambda R_{1234}\ge 0.
$$

\item A Riemannian manifold $(M^4,g)$ is said to have weakly \emph{PIC}$_2$ if for any $\lambda,\mu \in [-1,1]$, any $p \in M$ and any orthonormal 4-frame $\{e_1,e_2,e_3,e_4\}$ of $T_pM$, we have
$$
R_{1313}+\lambda^2 R_{1414}+\mu^2 R_{2323}+\lambda^2 \mu^2 R_{2424}-2\lambda \mu R_{1234}\ge 0.
$$

\item A Ricci flow solution $(M^n,g(t))_{t \in \mathcal{I}}$ is said to have weakly \emph{PIC} (uniformly \emph{PIC}, weakly \emph{PIC}$_1$, weakly \emph{PIC}$_2$) if $(M,g(t))$ has weakly \emph{PIC} (uniformly \emph{PIC}, weakly \emph{PIC}$_1$, weakly \emph{PIC}$_2$) for all $t \in \mathcal{I}$. For uniformly \emph{PIC}, the constant $\theta$ or $\Lambda$ is required to be uniform for all $t \in \mathcal{I}$.
\end{enumerate}
\end{defn}

For later applications, we need the following localization result similar to \cite[Proposition $2.1$]{C}.

\begin{prop} \label{P201}
For $r>0$ and $A \ge 14(n-1)r^{-2}T+2$, suppose $(M^n, g(t))_{t \in [0,T]}$ is a Ricci flow solution with a continuous function $f : M \times [0,T] \to \R$ satisfying the following properties.
\begin{enumerate}[label=\arabic*)]
    \item $B_{g(t)}(p,Ar)$ is compactly contained in $M$ for any $t \in [0,T]$.
    \item For any $t \in [0,T]$ and $x \in B_{g(t)}(p,r)$, $|\text{Ric}|(x,t)\le (n-1)r^{-2}$.
    \item For any $t \in [0,T]$ and $x \in B_{g(t)}(p,Ar)$, $\tl f(x,t) \ge \delta f^2(x,t)$ in the barrier sense for a constant $\delta>0$.
\end{enumerate}
Then there exists a constant $C=C(n)>0$ such that $f(x,t) \ge \min\left \{-\dfrac{4}{t\delta},-\dfrac{C}{(Ar)^2\delta}\right \}$ for any $t \in [0,T]$ and $x \in B_{g(t)}(p,\frac{3Ar}{4})$.
\end{prop}
\begin{proof}
 By rescaling, we may assume that $\delta=1$. As in \cite[Proposition $2.1$]{C}, there exists a cutoff function $\psi : M \times [0,T] \to \R$ that is compactly supported in $\cup_{t \in [0,T]}B_{g(t)}(p,Ar) \times \{t\}$ and $\psi(x,t)\equiv 1$ wherever $d_{g(t)}(p,x) \le \frac{3Ar}{4}$ and satisfies $\left|\tl \psi+\frac{2|\na \psi|^2}{\psi}\right|\le \frac{C}{(Ar)^2}\sqrt{\psi}$. Let $u(x,t)=\psi(x,t)f(x,t)$. Since there is nothing to prove if $f(x,t) \ge 0$ in $B_{g(t)}\left(p,\frac{3Ar}{4}\right) \times [0,T]$, we may assume that there exists $t_0 \in [0,T]$ such that $\sup_{x \in M}u(x,t_0)=u(x_0,t_0)<0$. If $(x_0,t_0)$ is a smooth point of both $d_{g(t)}(p,x)$ and $f$, then by using $0=\na u(x_0,t_0)=f\na \psi+\psi \na f$, we have at $(x_0,t_0)$,
\begin{align*}
    \tl u &=f\tl \psi+\psi \tl f+\frac{2f |\na \psi|^2}{\psi}\\
    &\ge \frac{Cf}{(Ar)^2}\sqrt{\psi}+\psi f^2 \ge \frac{\psi f^2}{2}-\frac{C^2}{2(Ar)^4}.
\end{align*}
From this, for $u_{\text{min}}(t):=\inf_{x \in M}u(x,t)$, we have
\begin{align*}
\frac{d^-}{dt}u_{\text{min}}(t_0)=&\liminf_{h \searrow0}\frac{u_{\text{min}}(t_0+h)-u_{\text{min}}(t_0)}{h} \\
&\ge \frac{u_{\text{min}}(t_0)^2}{2}-\frac{C^2}{2(Ar)^4}\\
&\ge \frac{u_{\text{min}}(t_0)^2}{4}+\left(\frac{u_{\text{min}}(t_0)^2}{4}-\frac{C^2}{2(Ar)^4}\right)
\end{align*}
and this inequality holds as long as $u_{\text{min}}(t_0)\le 0$. By integrating this inequality, we have
$$
u_{\text{min}}(t)\ge \min\left\{-\frac{4}{t},-\frac{\sqrt{2}C}{(Ar)^2}\right\}
$$
and it completes the proof since $\psi\equiv 1$ in $B_{g(t)}(p,\frac{3Ar}{4}) \times [0,T]$. If $d_{g(t)}(p,x)$ or $f$ is not smooth at $(x_0,t_0)$, then we can choose a barrier function and compute in a same way. \end{proof}

The following corollary is immediate.
\begin{cor} \label{C201}
Let $(M,g(t))_{t \in (-\infty,0]}$ be a complete ancient solution to the Ricci flow (not necessarily having bounded curvature). If a continuous function $f : M \times (-\infty,0] \to \R$ satisfies the inequality
$$\tl f \ge \delta f^2$$in the barrier sense for some $\delta>0$, then $f(x,t)\ge0$ for all $(x,t) \in M \times (-\infty,0]$.
\end{cor}
\begin{proof} 
For each $p \in M$, we can choose $r>0$ small enough so all assumptions in Proposition \ref{P201} are satisfied. By taking $A\to \infty$ and translating the initial time by $-\tau_0$, we get
$$
f(x,t) \ge -\frac{4}{(t+\tau_0)\delta}.
$$
Then the result is obtained by taking $\tau_0 \to \infty$. 
\end{proof}

In the following sections, we will apply Corollary \ref{C201} multiple times to obtain the curvature improvements.

\section{Higher dimensional curvature improvement} \label{sec:high}
In this section, we will show that every $n$-dimensional, complete ancient solution to the Ricci flow with $n\ge12$ and uniformly PIC automatically has weakly PIC$_2$. Notice that the bounded curvature is not assumed.

\subsection{Uniformly 2-positive Ricci curvature}
In \cite[Proposition $5.2$]{LN}, it is proved that if an ancient solution to the Ricci flow has weakly PIC, then it has 2-nonnegative Ricci curvature. Here, we will show the following uniform version, which will be used in the next subsection.
\begin{prop} \label{prop:2ric}
Let $(M^n,g(t))_{t \in (-\infty,0]}$ be a complete ancient solution to the Ricci flow with uniformly \emph{PIC} for $n \ge 5$. Then there exists $\delta =\delta(n,\theta)>0$ such that
$$
\lambda_1+\lambda_2 \ge \delta R
$$
on $M \times (-\infty,0]$, where $\lambda_1+\lambda_2$ is a sum of the two smallest eigenvalues of $Ric$.
\end{prop}
\begin{proof} Let $f=\lambda_1+\lambda_2-\delta R$ with $\delta>0$ determined later. Fix a spacetime point $(x,t)$, we choose an orthonormal frame $\{e_1,e_2,\cdots,e_n\}$ such that the Ricci curvature is diagonalized with eigenvalues $R_{11}\le R_{22}\le \cdots \le R_{nn}$. By applying the Uhlenbeck's trick, we have the following evolution equation
$$
\tl R_{ij}=2R_{ikjl}R_{kl}.
$$
Therefore, the function $f$ satisfies the following evolution equation in the barrier sense
\begin{align}
    \tl f& \ge 2(R_{1i1i}+R_{2i2i})R_{ii}-2\delta |\text{Ric}|^2. \label{eq:ric1a}
\end{align}
Now we investigate each term as follows.
On the one hand, from weakly PIC, we have $R_{ii} \ge -\frac{R}{n-4}$, see \cite[Proposition $7.3$]{RFST}. On the other hand, we obtain
$$
R-2R_{nn}=\left(\sum_{i=1}^{n-1}R_{ii}\right)-R_{nn}=\sum_{i,j=1}^{n-1}R_{ijij}\ge0.
$$ 
Therefore, we conclude that $|\text{Ric}|^2 \le nR^2$ and \eqref{eq:ric1a} becomes
\begin{align}
    \tl f& \ge 2(R_{1i1i}+R_{2i2i})R_{ii}-2n\delta R^2. \label{eq:ric1}
\end{align}
Also, we can rewrite
\begin{align}
2(R_{1i1i}+R_{2i2i})R_{ii}&=\sum_{i \ge 3}(R_{1i1i}+R_{2i2i})(2R_{ii}-R_{11}-R_{22})+(R_{11}+R_{22})^2 \label{eq:ric2}
\end{align}
and 
\begin{align*}
&\sum_{i \ge 3}(R_{1i1i}+R_{2i2i})(2R_{ii}-R_{11}-R_{22}) \\
=&2\theta R(2R-n(R_{11}+R_{22}))+\sum_{i \ge 3}(R_{1i1i}+R_{2i2i}-2\theta R)(2R_{ii}-R_{11}-R_{22}).
\end{align*}
If $f(x,t) \le 0$, then we additionally have
$$
2\theta R(2R-n(R_{11}+R_{22}))\ge 2\theta(2-n\delta)R^2.
$$ 
From the PIC condition, we know that there is at most one $i \ge 3$ such that $R_{1i1i}+R_{2i2i}-2\theta R <0$. Based on this observation, there are 3 possibilities.
\begin{enumerate}
    \item If $R_{1i1i}+R_{2i2i}-2\theta R \ge 0$ for all $i =3,\cdots,n$, then we have
\begin{align*}
\sum_{i \ge 3}(R_{1i1i}+R_{2i2i}-2\theta R)(2R_{ii}-R_{11}-R_{22})\ge 0.
\end{align*}
    \item If $R_{1i1i}+R_{2i2i}-2\theta R < 0$ for some $i =3,\cdots,n-1$, then we still have $$\sum_{i \ge 3}(R_{1i1i}+R_{2i2i}-2\theta R)(2R_{ii}-R_{11}-R_{22})\ge 0$$
    since
    \begin{align*}
        &(R_{1i1i}+R_{2i2i}-2\theta R)(2R_{ii}-R_{11}-R_{22})+(R_{1n1n}+R_{2n2n}-2\theta R)(2R_{nn}-R_{11}-R_{22})\\
        &\ge (R_{1n1n}+R_{2n2n}-2\theta R) (2R_{nn}-2R_{ii}) \ge 0
    \end{align*}
    \item If $R_{1n1n}+R_{2n2n}-2\theta R < 0$ and $f(x,t) \le 0$, then there exists a constant $C=C(n,\theta)>0$ such that
    \begin{align*}&\sum_{i \ge 3}(R_{1i1i}+R_{2i2i}-2\theta R)(2R_{ii}-R_{11}-R_{22})\\
    \ge& (2\theta R-R_{1n1n}-R_{2n2n})\left(\sum_{i=3}^{n-1}(2R_{ii}-R_{11}-R_{22})+R_{11}+R_{22}-2R_{nn}\right)\\
   =&(2\theta R-R_{1n1n}-R_{2n2n})(2R-4R_{nn}-(n-2)(R_{11}+R_{22}))  \ge -\delta CR^2\end{align*}
 where the last inequality holds since
$$
2R-4R_{nn}-(n-2)(R_{11}+R_{22}) \ge -(n-2)(R_{11}+R_{22}) \ge -(n-2)\delta R
$$
and
$$
2\theta R-R_{1n1n}-R_{2n2n}\le2\theta R+|R_{1n1n}|+|R_{2n2n}|\le (2\theta+C_1)R,
$$
where $C_1=C_1(n)>0$ since the curvature operator is controlled by the scalar curvature for weakly PIC condition.
\end{enumerate}
Therefore, in any case, we have
\begin{align}
    \sum_{i \ge 3}(R_{1i1i}+R_{2i2i})(2R_{ii}-R_{11}-R_{22})\ge 2\theta(2-n\delta)R^2-\delta CR^2\ge 3\theta R^2  \label{eq:ric3}
    \end{align}
    if $f(x,t) \le 0$ by taking $\delta>0$ small enough.
By combining \eqref{eq:ric1}, \eqref{eq:ric2} and \eqref{eq:ric3}, we finally obtain that if $f(x,t) \le 0$,
$$\tl f\ge (R_{11}+R_{22})^2+3\theta R^2-2n\delta R^2 \ge (R_{11}+R_{22})^2+\delta^2R^2$$
by decreasing $\delta>0$ if necessary. It implies that
$$\tl f^- \ge \frac{1}{2}(f^-)^2$$
where $f^- := \min\{f,0\}$.
Now the proof is complete by Corollary \ref{C201}. \end{proof}

\subsection{Weakly PIC$_2$ condition}

Now we prove the main theorem in the section. 
\begin{thm} \label{thm:pic2}
Let $(M^n,g(t))_{t \in (-\infty,0]}$ be a complete ancient solution to the Ricci flow with uniform \emph{PIC} for $n \ge 12$. Then it has weakly \emph{PIC}$_2$.
\end{thm}
\begin{proof} We first recall that by applying Uhlenbeck's trick, the Riemannian curvature under the Ricci flow is deformed by
\begin{align}
\tl R_{ijkl}=Q(\Rm)_{ijkl} :=R_{ijpq}R_{klpq}+2R_{ipkq}R_{jplq}-2R_{iplq}R_{jpkq}. \label{eq:riem}
\end{align}
In the proof, we will work with a local orthonormal basis given as follows. For each point $(x_0,t_0)$, we choose an orthonormal basis $\{e_1,\cdots,e_n\}$ with respect to $g(t)$. After one applies the Uhlenbeck's trick by taking a pullback of $g(t)$ using a 1-parameter family of bundle isomorphism. Then a pullback of $g(t)$ is independent to $t$. Therefore, $\{e_1,\cdots,e_n\}$ is still an orthonormal basis at $(x_0,t)$ for any $t$ with respect to this pullback. By parallel transport, there exists an orthonormal frame on a spacetime neighborhood of $(x_0,t_0)$. With the help of such a frame, locally around $(x_0,t_0)$, we can consider $\Rm(x,t)$ in a fixed vector space of algebraic curvature tensor $\mathscr C_{B}(\R^n)$.

In \cite[Definition 3.1]{B1} and \cite[Definition 4.1]{B1}, Brendle has constructed two continuous families of closed, convex, $O(n)$-invariant cones $\mathcal{C}(b)$ and $\tilde{\mathcal{C}}(b)$ in $\mathscr C_{B}(\R^n)$ for $n \ge 12$. Moreover, the following properties hold:
\begin{itemize}
    \item $\mathcal{C}(b)$ is defined for all $b \in (0,b_{\text{max}}]$.
    \item $\tilde{\mathcal{C}}(b)$ is defined for all $b \in (0,\tilde{b}_{\text{max}}]$.
    \item $\lim_{b \to 0}\tilde{\mathcal{C}}(b) = \mathcal{C}(b_{\text{max}}) \cap \text{PIC}_1$. 
    \item $\mathcal{C}(b_{\text{max}})=\tilde{\mathcal{C}}(\tilde{b}_{\text{max}})$.
\end{itemize}
As in \cite[Section $5$]{B1}, we define
$$\hat{\mathcal{C}}(b):=\begin{cases} \mathcal{C}(b) & \text{if}~b \in [0,b_\text{max})\\ \tilde{\mathcal{C}}(b_\text{max}+\tilde{b}_\text{max}-b)&\text{if}~b \in [b_\text{max},b_\text{max}+\tilde{b}_\text{max}).\end{cases}$$

From the construction, $I=\frac{1}{2} \text{id} \owedge \text{id}=\delta_{ik}\delta_{jl}-\delta_{il}\delta_{jk}$ is in the interior of $\hat{\mathcal{C}}(b)$, where $\owedge$ is the Kulkarni-Nomizu product. Moreover, if $\Rm \in \hat{\mathcal{C}}(b)$, then $Q(\Rm)$ lies in the interior of $T_{\Rm}\hat{\mathcal{C}}(b)$, see \cite[Theorem 3.2]{B1}. Here $T_{\Rm}\hat{\mathcal{C}}(b)$ is the tangent cone to $\hat{\mathcal{C}}(b)$ at $\Rm$, see \cite[Definition $5.1$]{RFST}.

\textbf{Claim 1.} There exists $b_0>0$ such that the curvature tensor of $(M,g(t))$ is contained in $\hat{\mathcal{C}}(b_0)$ for all $t \in (-\infty,0]$. 

First, if $\Rm$ has uniformly PIC, then the curvature operator is controlled by the scalar curvature. That is, there exists a $C=C(n)>0$ such that
\begin{align}
|\Rm| \le CR. \label{eq:311}
\end{align}

Next we recall the transformation $l_{a,b}$, defined in \cite{BW}, so that there exists $S \in \mathscr C_{B}(\R^n)$ such that
\begin{align}
l_{a,b}(S):=S+b\text{Ric}(S)\owedge \text{id} +\left(\frac{2(a-b)}{n} \text{scal}(S) \right) I=\Rm. \label{eq:312}
\end{align}

Therefore, if we choose $T=\frac{\theta}{2}RI$, then it is clear from \eqref{eq:311}, \eqref{eq:312} and Proposition \ref{prop:2ric} that $\Rm \in \hat{\mathcal{C}}(b_0)$ if $b_0$ is sufficiently small, see \cite[Definition $3.1$]{B1}. Here, the condition (iii) in \cite[Definition $3.1$]{B1} is guaranteed by Proposition \ref{prop:2ric}.

Now we prove that if the curature operator $\Rm$ of $(M,g(t))_{t \in (-\infty,0]}$ is contained in $\hat{\mathcal{C}}(b)$ for some $b \in (0,b_\text{max}+\tilde{b}_\text{max})$, then $\Rm \in \hat{\mathcal{C}}(b')$ if $b'$ is sufficiently close to $b$.

\textbf{Claim 2.} There exists a constant $\tau \in (0,1)$ depending only on $n$ and $b$ such that $Q(\text{Rm})-\tau R^2 I \in T_{\Rm}\hat{\mathcal{C}}(b)$.

If it is not true, then since $\hat{C}(b)$ is a cone and $R^{-2}Q(\Rm)$ is scaling invariant, one can choose a sequence of counterexamples $\Rm_k\in \del \hat{C}(b)$ such that $|\Rm_k|=1$ and $Q(\Rm_k)-k^{-2}R_k^2I$ is on the boundary of $T_{\Rm_k}\hat{\mathcal{C}}(b)$ for all $k>0$. By taking a subsequence,  $\Rm_k$ converges to $\Rm_\infty \in \del \hat{\mathcal{C}}(b)$ and $Q(\Rm_\infty)$ is on the boundary of $T_{\Rm_\infty}\hat{\mathcal{C}}(b)$. However, this contradicts the transversality of $\hat{\mathcal{C}}(b)$.

Now we define a function $\lambda$. For any spacetime point $(x,t)$, let $\lambda(x,t)$ be the smallest number so that the curvature operator $S:=\Rm+(\lambda-\delta R)I$ lies on the boundary of $\hat{C}(b)$, where the constant $\delta>0$ is determined later. Here we assume $\lambda$ is locally smooth and the general case will be handled at the end of the proof. For any $(x,t) \in M \times (-\infty,0]$, $\Rm(x,t)$ is contained in $\hat{\mathcal{C}}(b)$ and hence $\lambda(x,t)-\delta R(x,t) \le 0$. From the direct computation, we have
$$
Q(S)=Q(\text{Rm})+2(\lambda-\delta R)\text{Ric}\owedge \text{id}+2(n-1)(\lambda-\delta R)^2I.
$$
Therefore, we can derive the following evolution equation
\begin{align}
    \tl S&=Q(\text{Rm})+\left(\tl \lambda\right)I-2\delta |\text{Ric}|^2I \notag\\
    &=Q(S)-2(\lambda-\delta R)\text{Ric}\owedge \text{id}-2(n-1)(\lambda-\delta R)^2I-2\delta |\text{Ric}|^2I+\left(\tl \lambda\right)I.  \label{eq:312a}
\end{align}
Now, wherever $\lambda(x,t) \ge 0$, we know that
$$\text{scal}(S)=R+n(n-1)(\lambda-\delta R)\ge \frac{R}{2}$$
for small enough $\delta>0$ and $$|\lambda-\delta R|=\delta R-\lambda\le \delta R.$$

From \eqref{eq:311}, we have
\begin{align*}
-2(\lambda-\delta R)\text{Ric}\owedge \text{id}-2(n-1)(\lambda-\delta R)^2I-2\delta |\text{Ric}|^2I \ge-C_1\delta R^2 I
\end{align*}
for a constant $C_1=C_1(n)>0$.
Therefore, we obtain from \eqref{eq:312a}
$$
\tl S \ge \left(\tl \lambda \right)I+Q(S)-C_1\delta R^2I,
$$
wherever $\lambda(x,t) \ge 0$.
Also, since $\lambda-\delta R \le 0$ and $R>0$, if $\lambda(x,t) \ge 0$, then we have $$\lambda^2(x,t) \le \delta^2 R^2(x,t)$$ which implies that
$$-C_1\delta R^2-\lambda^2\ge -(C_1\delta+\delta^2)R^2 \ge -C_2\delta \text{scal}^2(S),$$ where $C_2=4(C_1+1)$. As a result, we obtain
\begin{align}
\tl S \ge \left(\tl \lambda+\lambda^2\right)I+Q(S)-C_2\delta \text{scal}^2(S) \label{E311}
\end{align}
wherever $\lambda(x,t) \ge 0$.
Now we observe the following. 

\textbf{Claim 3.} $\tl S$ is not contained in the interior of $T_S\hat{\mathcal{C}}(b)$.

To show the Claim 3, we fix a point $(x_0,t_0) \in M \times (-\infty,0]$ and choose a supporting plane $H$ of $\hat{\mathcal{C}}(b)$ at $S(x_0,t_0)$ with a normal vector $\nu$ pointing toward a half-space containing $\hat{\mathcal{C}}(b)$. Then we can see that $F(p)=\left<p,\nu\right>$ is a distance function between $p$ and $H$. So the function $F(S(x_0,t))$ has its local minimum at $t=t_0$. It implies that
$$
0=\left.\frac{\del}{\del t} F(S(x_0,t))\right|_{t=t_0}=\left< \frac{\del}{\del t} S(x_0,t_0),\nu \right>.
$$
Also, if we choose any tangent vector $u \in T_{x_0}M$ and consider a geodesic $\gamma(s)$ starting from $x_0$ with a directional vector $u$, then $F(S(\gamma(s),t_0))$ has its local minimum at $s=0$. So we have
$$
0\le \left.\frac{d^2}{ds^2}F(S(\gamma(s),t_0))\right|_{s=0}=\left< \na_u \na_u S(x_0,t_0),\nu \right>.
$$
After taking the sum for $u$, we get $\left< \Delta S(x_0,t_0),\nu \right>\ge0$. By combining this two results, we get
$$\left< \tl S(x_0,t_0),\nu \right> \le 0$$
which shows that $\tl S(x_0,t_0)$ can not be contained in the interior of $T_S\hat{\mathcal{C}}(b)$. It verifies the Claim 3.

Since $S$ is contained of $\hat{\mathcal{C}}(b)$, we know from Claim 2 that $Q(S)-\tau \text{scal}^2(S) I \in T_S\hat{\mathcal{C}}(b)$. Therefore, if we choose $\delta>0$ small enough so that $\delta<\frac{\tau}{2C_2}$, then $Q(S)-C_2\delta \text{scal}^2(S)I$ is contained in the interior of $T_S\hat{\mathcal{C}}(b)$. Combining this with Claim 3, we conclude that the first term $\left(\tl \lambda+\lambda^2 \right)I$ on the right side of \eqref{E311} should not be contained in $T_S\hat{\mathcal{C}}(b)$. Therefore, we obtain
\begin{align}
\tl\lambda+\lambda^2 \le 0 \label{E311aa}
\end{align}
wherever $\lambda(x,t) \ge 0$. By applying Corollary \ref{C201} on $(-\lambda)^-$, we obtain $\lambda(x,t) \le 0$ for all $(x,t) \in M \times (-\infty,0]$. In sum, we have shown that $\Rm-\delta R I \in \hat{\mathcal{C}}(b)$ on $M \times (-\infty,0]$. Therefore, $\Rm \in \hat{\mathcal{C}}(b')$ if $b'$ is close to $b$.

Now, let us consider the case when $\lambda$ is not smooth at $(x_0,t_0)$. Note that $S(x_0,t_0)=\Rm(x_0,t_0)+(\lambda(x_0,t_0)-\delta R(x_0,t_0))I$ is on the boundary of $\hat{\mathcal{C}}(b)$. So we can choose a supporting hyperplane $H$ of $\hat{\mathcal{C}}(b)$ at $S(x_0,t_0)$. Using this hyperplane, we can define a new function $\tilde{\lambda}$ which is the smallest number so that $\tilde{S}:=\Rm+(\tilde{\lambda}-\delta R)I$ lies on $H$. Clearly, $\tilde{\lambda}(x,t)$ is smooth and $\tilde{\lambda}(x_0,t_0)=\lambda(x_0,t_0)$. Also, since $I$ is in the interior of $\hat{\mathcal{C}}(b)$, we have $\tilde{\lambda}(x,t) \le \lambda(x,t)$ for all $(x,t)$ in a small neighborhood of $(x_0,t_0)$. In other words, the function $\tilde{\lambda}$ is a lower barrier function of $\lambda$ from below. Also, since $F(p)=\left< p,\nu\right>$ is constant for $p \in H$, we can use the proof of Claim 3 again to show that $\tl \tilde{S}(x_0,t_0)$ has to be contained in $H$. From the same argument, $\tilde{\lambda}$ satisfies the differential inequality
$$\tl \tilde{\lambda}+\tilde{\lambda}^2 \le 0$$
at $(x_0,t_0)$, wherever $\tilde{\lambda}(x_0,t_0) \ge 0$. Therefore, \eqref{E311aa} holds in the barrier sense, which is sufficient to apply Corollary \ref{C201}.

Now we set 
$$
\mathcal{I}:=\{b \in (0,b_\text{max}+\tilde{b}_\text{max})~|~\Rm(x,t)~\text{is contained in }\hat{\mathcal{C}}(b)~\text{for all }(x,t) \in M \times (-\infty,0]\}.
$$
Then $\mathcal{I}$ is nonempty since $b_0 \in \mathcal{I}$ and open from the previous argument. Also, it is closed since $\hat{\mathcal{C}}(b)$ is a continuous family. Therefore, it implies that $\Rm(x,t)$ is contained in $\hat{\mathcal{C}}(b)$ for all $b \in (0,b_\text{max}+\tilde{b}_\text{max})$. By taking $b \to (b_\text{max}+\tilde{b}_\text{max})$, we can conclude that $(M,g(t))_{t \in (-\infty,0]}$ has weakly $\mathrm{PIC}_1$. Then by \cite[Lemma 4.2]{BCW}(see also \cite[Proposition 6.2]{LN}), we conclude that $(M,g(t))_{t \in (-\infty,0]}$ has weakly $\mathrm{PIC}_2$.

In sum, the proof of Theorem \ref{thm:pic2} is complete. 
\end{proof}

In general, we can follow \cite{BHS11} (see also \cite{Ka19}) and consider a cone $C$ in $\mathscr C_{B}(\R^n)$ with $(*)$ conditon defined as follows:
\begin{itemize}
    \item $C$ is closed, convex, $O(n)$-invariant, and of full-dimension.
    \item $C$ is transversally invariant under the Hamilton's ODE: $\dfrac{d \Rm}{dt}=Q(\Rm)$.
    \item Every algebraic curvature tensor $\Rm \in C \backslash \{0\}$ has positive scalar curvature.
    \item The identity $I$ lies in the interior of $C$.
\end{itemize}

Now the same argument of Theorem \ref{thm:pic2} proves the following theorem, which is motivated by \cite[Theorem $9$]{BHS11} on ancient solutions to the compact Ricci flow.

\begin{thm}\label{thm:cones}
Let $C(s),\,s\in [0,1],$ be a continuous family of cones in $\mathscr C_{B}(\R^n)$ with $(*)$ condition. Suppose $(M^n,g(t))_{t \in (-\infty,0]}$ is a complete ancient solution to the Ricci flow such that $\emph{Rm}(x,t) \in C(0)$ for all points $(x,t) \in M \times (-\infty,0]$. Then $\emph{Rm}(x,t) \in C(1)$ for all $(x,t) \in M \times (-\infty,0]$.
\end{thm}

\section{4-dimensional curvature improvement}
\subsection{Ricci flow with uniformly PIC}
In this subsection, we consider a complete $4$-dimensional ancient solution to the Ricci flow with uniformly PIC and prove the curvature improvement. Let us remark that the same result is given in \cite[Theorem $1.3$]{B3} for steady solitons. However, unlike the steady soliton case, the boundedness of curvature is not obtained automatically. Therefore, we rely on Corollary \ref{C201} to obtain the result.

First, we recall some notations introduced in \cite{H}. After using the self-dual and anti-self-dual decomposition of $\bigwedge^2 \R^4=\bigwedge_+ \oplus \bigwedge_-$, we can write the curvature operator as 
$$
\Rm=\begin{pmatrix}
A & B \\ B^t & C \end{pmatrix}.
$$
Moreover, let $a_1 \le a_2 \le a_3$ and $c_1 \le c_2 \le c_3$ be eigenvalues of $A$ and $C$, respectively. Also, we denote the eigenvalues of the symmetric matrix $\sqrt{BB^t}$ by $0 \le b_1 \le b_2 \le b_3$. It is clear by the Bianchi identity that
$$
\text{tr}(A)=a_1+a_2+a_3=\frac{R}{2}=c_1+c_2+c_3=\text{tr}(C).
$$
In addition, we recall that being weakly PIC is equivalent to $\min \{a_1+a_2,c_1+c_2\} \ge 0$, see \cite{H1}.

Next, we prove the following lemma.

\begin{lem} \label{lem:400}
Let $(M^4,g(t))_{t \in (-\infty,0]}$ be a complete 4-dimensional ancient solution to the Ricci flow with weakly \emph{PIC}. Then we have 
$$
a_1 \ge 0 \quad \text{and} \quad c_1\ge 0
$$
on $M \times (-\infty,0]$.
\end{lem}
\begin{proof} From the evolution equation of the curvature operator (see \cite{H1}), we have
\begin{align*}
&\tl a_1 \ge a_1^2+b_1^2+2a_2a_3 \ge a_1^2, \\
&\tl c_1 \ge c_1^2+b_1^2+2c_2c_3 \ge c_1^2.
\end{align*}
Now the result follows from Corollary \ref{C201}.
\end{proof}

Next, we show that $a_3$ and $c_3$ are controlled solely by $a_1$ and $c_1$, respectively.
\begin{lem} \label{lem:401}
Let $(M^4,g(t))_{t \in (-\infty,0]}$ be a complete 4-dimensional ancient solution to the Ricci flow with uniformly \emph{PIC}. Then we have
$$a_3 \le (6\Lambda^2+1)a_1$$
$$c_3 \le (6\Lambda^2+1)c_1$$
on $M \times (-\infty,0]$.
\end{lem}
\begin{proof} 
From direct calculations,
\begin{align*}
    &\tl (a_3-(6\Lambda^2+1)a_1) \\
\le& \, a_3^2+b_3^2+2a_1a_2-(6\Lambda^2+1)(a_1^2+b_1^2+2a_2a_3)\\
    \le&\, a_3^2+b_3^2 -12\Lambda^2a_2a_3 \le a_3^2+b_3^2-3\Lambda^2(a_1+a_2)^2 \le-a_3^2.
\end{align*}
since $\max(a_3,b_3)\le \Lambda(a_1+a_2)$ and $(a_1+a_2)^2 \le 4a_2^2\le 4a_2a_3$. Moreover, if we have $(a_3-(6\Lambda^2+1)a_1)(x_0,t_0)\ge 0$ for some $(x_0,t_0) \in M \times (-\infty,0]$ , then we have at $(x_0,t_0)$,
\begin{align*}
    \tl (a_3-(6\Lambda^2+1)a_1)&\le -(a_3-(6\Lambda^2+1)a_1)^2.
\end{align*}
Now the result follows from Corollary \ref{C201}, by choosing $f=(-a_3+(6\Lambda^2+1)a_1)^-$. Similarly, the conclusion for $c_1$ and $c_3$ also holds.
\end{proof}

We continue to show the following. 

\begin{lem}\label{lem:402}
Let $(M^4,g(t))_{t \in (-\infty,0]}$ be a complete 4-dimensional ancient solution to the Ricci flow with uniformly \emph{PIC}.  Then we have $$\frac{b_3^2}{(a_1+a_2)(c_1+c_2)} \le \frac14$$ on $M \times (-\infty,0]$.
\end{lem}
\begin{proof} If it is not true, then from the boundedness of this ratio, we have
$$\gamma:=\sup_{M \times (-\infty,0]} \frac{b_3}{\sqrt{(a_1+a_2)(c_1+c_2)}}>\frac12 .$$
From \cite{H}, we have the following evolution equations
\begin{align*}
    \tl b_3 \le& b_3(a_3+c_3)+2b_1b_2, \\
    \tl (a_1+a_2)\ge& a_1^2+a_2^2+2a_3(a_1+a_2)+b^2_1+b_2^2, \\
  \tl (c_1+c_2)\ge& c_1^2+c_2^2+2c_3(c_1+c_2)+b^2_1+b_2^2.
\end{align*}

From direct calculations, we derive
\begin{align*}
    \tl \sqrt{(a_1+a_2)(c_1+c_2)}\ge \sqrt{(a_1+a_2)(c_1+c_2)}(G+E)
\end{align*}
where
\begin{align*}
    G=\frac14 |\na \log(a_1+a_2)-\na \log(c_1+c_2)|^2 \ge 0
\end{align*}
and
\begin{align*}
    E&=\frac12 \left(\frac{a_1^2+a_2^2+b_1^2+b_2^2}{a_1+a_2}+\frac{c_1^2+c_2^2+b_1^2+b_2^2}{c_1+c_2}\right)+a_3+c_3\\
    &=\frac12 \left(\frac{(a_1-b_1)^2+(a_2-b_2)^2}{a_1+a_2}+\frac{(c_1-b_1)^2+(c_2-b_2)^2}{c_1+c_2}\right)+a_3+c_3+2b_1+\frac{a_2(b_2-b_1)}{a_1+a_2}+\frac{c_2(b_2-b_1)}{c_1+c_2}.
\end{align*}
If we set $u(x,t)=b_3-\gamma\sqrt{(a_1+a_2)(c_1+c_2)}$, then $u(x,t)\le0$ everywhere and it satisfies
\begin{align*}
    \tl u(x,t)&\le b_3(a_3+c_3)+2b_1b_2-\gamma\sqrt{(a_1+a_2)(c_1+c_2)}E\\
    &\le u(x,t)(a_3+c_3)+2b_1(b_2-\gamma\sqrt{(a_1+a_2)(c_1+c_2)})-\gamma\sqrt{(a_1+a_2)(c_1+c_2)}K
\end{align*}
where 
$$
K=\frac12 \left(\frac{(a_1-b_1)^2+(a_2-b_2)^2}{a_1+a_2}+\frac{(c_1-b_1)^2+(c_2-b_2)^2}{c_1+c_2}\right)+\frac{a_2(b_2-b_1)}{a_1+a_2}+\frac{c_2(b_2-b_1)}{c_1+c_2}.
$$
From the expression, it is clear that $\tl u(x,t) \le 0$ everywhere. Moreover, it can be zero only when $K=u(x,t)=0$ and $b_2=b_3$. But this implies that $a_1=b_1=c_1=a_2=b_2=c_2=b_3$ and 
$$0=u(x,t)=b_3-\gamma\sqrt{(a_1+a_2)(c_1+c_2)}=(1-2\gamma)b_3$$
Since $\gamma>\frac12$ from the assumption, it can only happen when all terms are equal to 0, which is impossible since $a_1+a_2>0$. Moreover, it follows from the uniformly PIC and Lemma \ref{lem:401} that $\tl u(x,t)\le-6\delta|\text{Ric}|^2$ for a small constant $\delta=\delta(\Lambda,\gamma)>0$. With this fact, we obtain
$$
\tl (u+\delta R)\le-4\delta |\text{Ric}|^2\le -\delta R^2.
$$
Therefore, we have
$$
\tl (u+\delta R)\le -\delta^{-1} (u+\delta R)^2
$$
wherever $u+\delta R \ge 0$. From Corollary \ref{C201}, it implies that $u(x,t)+\delta R(x,t)\le0$ everywhere. From this, we have
$$
0\ge \frac{b_3}{\sqrt{(a_1+a_2)(c_1+c_2)}}-\gamma+\frac{\delta R}{\sqrt{(a_1+a_2)(c_1+c_2)}}\ge \frac{b_3}{\sqrt{(a_1+a_2)(c_1+c_2)}}-\gamma+\delta'
$$
for a small constant $\delta'=\delta'(\Lambda,\gamma)>0$. However, it implies that $\frac{b_3}{\sqrt{(a_1+a_2)(c_1+c_2)}}\le \gamma-\delta'$ everywhere, which contradicts the choice of $\gamma$. So we conclude that $\gamma \le \frac12$, which completes the proof. \end{proof}

One can check that Lemma \ref{lem:402} actually implies $(M,g(t))_{t \in (-\infty,0]}$ has weakly PIC$_1$. However, in 4-dimensional case, we can get a stronger result as follows.
\begin{lem} \label{lem:403}
Let $(M^4,g(t))_{t \in (-\infty,0]}$ be a 4-dimensional, complete ancient solution to the Ricci flow with uniformly \emph{PIC}. Then it has nonnegative curvature operator.
\end{lem}
\begin{proof} To this end, we will show that
$$\frac{b_3^2}{a_1c_1} \le 1$$ for all $(x,t) \in M \times (-\infty,0]$. If it is not true, then we have 
$$\eta:=\sup_{M \times (-\infty,0]}\frac{b_3}{\sqrt{a_1c_1}}>1.$$
Notice that $\eta$ is finite by Lemma \ref{lem:401}. If we set $v(x,t)=b_3-\eta \sqrt{a_1c_1}$, then we get $v(x,t) \le 0$ everywhere and we have the following evolution equation.
\begin{align*}\tl v(x,t) &\le b_3(a_3+c_3)+2b_1b_2-\eta \sqrt{a_1c_1}E
\end{align*}
where
\begin{align*}
    E&=\frac12\left(\frac{a_1^2+b_1^2+2a_2a_3}{a_1}+\frac{c_1^2+b_1^2+2c_2c_3}{c_1}\right)\\
    &=\frac12 \left(\frac{(a_1-b_1)^2+2a_3(a_2-a_1)}{a_1}+\frac{(c_1-b_1)^2+2c_3(c_2-c_1)}{c_1}\right)+2b_1+a_3+c_3\\
    &=:F+2b_1+a_3+c_3.
\end{align*}
Therefore, we obtain
$$\tl v(x,t) \le v(x,t)(a_3+c_3)+2b_1(b_2-\eta\sqrt{a_1c_1})-\eta\sqrt{a_1c_1}F\le 0$$ since $F\ge0$ and $v(x,t) \le 0$. Moreover, the equality case is obtained if and only if $a_1=b_1=c_1=a_2=c_2$, $b_2=b_3$ and $b_3=\eta\sqrt{a_1c_1}$. Since we already know that $\frac{b_3}{\sqrt{(a_1+a_2)(c_1+c_2)}}\le \frac12$ from Lemma \ref{lem:402}, we get $b_3\le \frac12 \sqrt{(a_1+a_2)(c_1+c_2)}=b_1$. Therefore the equality case is obtained when $a_1=b_1=c_1=a_2=c_2=b_2=b_3$ and $b_3=\eta\sqrt{a_1c_1}>\sqrt{a_1c_1}=b_3$. It shows that $a_1=c_1=0$ in the equality case, which contradicts to the curvature condition. As before, $\tl v(x,t)\le -6\delta|\text{Ric}|^2$ for a small constant $\delta=\delta(\Lambda, \eta)>0$. From the similar argument used in Lemma \ref{lem:402}, we obtain a contradiction. 

Now we show that $\frac{b_3}{\sqrt{a_1c_1}}\le 1$ implies the nonnegativity of the curvature operator. To do so, let $\varphi_i^\pm \in \bigwedge^2_{\pm}$ $(i=1,2,3)$ be bases of a self-dual vector space $\bigwedge_+$ and an anti-self-dual vector space $\bigwedge_-$ that $\text{Rm}|_{\bigwedge^2_\pm}$ is diagonalized, respectively. Let $\varphi=\sum_{i=1}^3p^i\varphi_i^++\sum_{j=1}^3q^j\varphi_j^-$. From the definition, we have $\text{Rm}(\varphi_i^+,\varphi_i^+)\ge a_1$ and $\text{Rm}(\varphi_i^-,\varphi_i^-)\ge c_1$ and $\text{Rm}(\varphi_i^+,\varphi_j^-)\ge -b_3$ for all $i,j=1,2,3$. Therefore we have
\begin{align*}
    \text{Rm}(\varphi,\varphi)&\ge 3a_1\sum_{i=1}^3(p^i)^2+3c_1\sum_{j=1}^3(q^j)^2-6b_3\sum_{i,j=1}^3|p^i||q^j|\\
    &\ge 3a_1\sum_{i=1}^3(p^i)^2+3c_1\sum_{j=1}^3(q^j)^2-6\sqrt{a_1c_1}\sum_{i,j=1}^3|p^i||q^j|\\
    &=3\sum_{i,j}^3(|p^i|\sqrt{a_1}-|q_j|\sqrt{c_1})^2 \ge 0.
\end{align*}
Therefore, the proof is complete. 
\end{proof}

In sum,  let $(M^4,g(t))_{t \in (-\infty,0]}$ be a 4-dimensional, complete, noncompact ancient solution to the Ricci flow with uniformly PIC. Then it follows from Lemma \ref{lem:401} and Lemma \ref{lem:403} that that there exists a constant $K=K(\Lambda)>0$ such that 
\begin{align}
a_3 \le Ka_1,~c_3\le Kc_1,~b_3^2\le a_1c_1. \label{eq:rest}
\end{align}
In other words, it satisfies the restricted isotropic curvature pinching condition in \cite{CZ}. 
\begin{rem}
We have learned that Brendle and Naff  \emph{\cite[Proposition A.2]{BN}} recently proved the same result in a different method by constructing a family of continuous cones, provided that the curvature is uniformly bounded for the ancient solution. Notice that by Theorem \ref{thm:cones} the curvature assumption is not necessary.
\end{rem}

\subsection{K\"ahler Ricci flow with weakly PIC}

In this subsection, we consider a complete (complex) 2-dimensional ancient solution to the \ka Ricci flow with weakly PIC.

For a \ka surface, we can choose a positively oriented orthonormal basis as $\{e_1,Je_1,e_2,Je_2\}$ where $J$ is a complex structure. They generate self-dual and anti-self-dual  two forms as
we choose a basis for $\bigwedge_+$ and $\bigwedge_-$ as 
\begin{align*}\varphi_1^\pm=e_1\wedge Je_1\pm e_2 \wedge Je_2\\ \varphi_2^\pm=e_1\wedge e_2\pm Je_2 \wedge Je_1\\ \varphi_3^\pm=e_1\wedge Je_2\pm Je_1 \wedge e_2\end{align*}
Using this basis and \ka condition together with Bianchi identity, one can represent the curvature operator as a $6 \times 6$ matrix
$$
\begin{pmatrix} \frac{R}{2} & 0 & 0 & {\rho}_1 & {\rho}_2 & {\rho}_3\\ 0&0&0&0 & 0 & 0 \\ 0&0&0&0 & 0 & 0 \\ {\rho}_1 & 0 & 0 &  &  &  \\ {\rho}_2 & 0 & 0 &  & C &  \\ {\rho}_3 & 0 & 0 &  &  & 
\end{pmatrix}
$$
with a $3 \times 3$ matrix $C$. Therefore, a \ka surface has weakly PIC if and only if $c_1+c_2 \ge 0$, where, as before, $c_1\le c_2\le c_3$ are eigenvalues of $C$.

Now we prove the main result of the subsection, see also \cite[Lemma $3.1$]{CTZ}.

\begin{lem} \label{lem:405}
Let $(M,g(t))_{t \in (-\infty,0]}$ be a complete, complex 2-dimensional ancient solution to K\"{a}hler-Ricci flow with weakly \emph{PIC}. Then it has nonnegative curvature operator.
\end{lem}

\begin{proof}
From Lemma \ref{lem:400}, we obtain $c_1 \ge 0$ on $M \times (-\infty,0]$. Moreover, we have the evolution equation
\begin{align}
\tl\Rm=\Rm^2+\Rm^\sharp=\Rm^2+2\begin{pmatrix} 0 & 0 \\ 0 & C^\sharp\end{pmatrix},\label{eq:411}
\end{align}
where $C^\sharp$, which is the adjoint matrix of $C$, has eigenvalues $\{c_1c_2,c_2c_3,c_3c_1\}$. If we set $\lambda$ to be the smallest eigenvalue of $\Rm$, then it follows from \eqref{eq:411} that
\begin{align*}
\tl \lambda \ge \lambda^2.
\end{align*}
Therefore, we conclude from Corollary \ref{C201} that $\Rm \ge 0$ on $M \times (-\infty,0]$.
\end{proof}

\section{Proof of Theorem \ref{thm:main1}}

First, we recall the following definition, see \cite[Definition $1.1$]{BN}.

\begin{defn}\label{def:kappa} \emph{[}$\kappa$-solutions to the Ricci flow\emph{]}
For $n \ge 4$, a complete noncompact ancient solution $(M^n,g(t))_{t  \in (-\infty,0]}$ to the Ricci flow is called a $\kappa$-solution if it satisfies (1) uniformly \emph{PIC}, (2) weakly \emph{PIC}$_2$, (3) $\kappa$-noncollapsed and (4) uniformly bounded curvature.
\end{defn}

All $\kappa$-solutions are completely classified in \cite{BN}.

\begin{thm}[Corollary $1.6$ of \cite{BN}] \label{thm:kappa}
Any $\kappa$-solution $(M^n,g(t))_{t  \in (-\infty,0]}$ is isometric to either a family of shrinking cylinders (or a quotient thereof ) or to the Bryant soliton.
\end{thm}

Based on the $\kappa$-solutions, we prove a canonical neighborhood theorem on Ricci flows with possibly unbounded curvature. We first recall the following result from \cite[Corollary $11.6$]{P}.

\begin{lem}[Perelman]\label{lem:pe}
For every $w>0$, there exist constants $C = C(w)<\infty$ and $\tau=\tau(w)>0$ with the following properties.
Let $(M^n,g(t))_{t\in [-T,0]}$ be a (possibly incomplete) Ricci flow solution with weakly \emph{PIC}$_2$. Suppose $B_{g(0)}(x_0,r_0)$ is compactly contained in $M$ such that $\emph{Vol}\,B_{g(0)}(x_0,r_0) \ge w r_0^n$ and $T\ge 2  \tau r_0^2$. Then
\begin{align*}
R(x,t) \le C r_0^{-2}
\end{align*}
for $(x,t) \in B_{g(0)}(x_0,r_0/4) \times [-\tau r_0^2,0]$.
\end{lem}
\begin{proof}
The proof is almost identical with the proof of \cite[Corollary $11.6$]{P}. The only difference is we use \cite[Lemma $4.5$]{CW} instead of \cite[Proposition $11.4$]{P}.
\end{proof}

\begin{thm} \label{thm:cano}
Let $(M^n,g(t))_{t  \in [0,2]}$ be a complete noncompact $\kappa$-noncollapsed Ricci flow solution with uniformly \emph{PIC} and weakly \emph{PIC}$_2$. For any $\ep>0$, there exists a small number $\bar r>0$ satisfying the following property.

Suppose $(\bar x, \bar t) \in M \times [1,2]$ and $R(\bar x, \bar t)=r^{-2} \ge {\bar r}^{-2}$, then after rescaling the metric by the factor $r^{-2}$, the parabolic neighborhood $B_{g(\bar{t})}(\bar{x},\ep^{-1} r) \times [\bar{t}-\ep^{-1} r^2,\bar{t}]$ is $\ep$-close in $C^{[\ep^{-1}]}$-topology to a $\kappa$-solution.
\end{thm}

\begin{proof}
The proof is similar to the canonical neighborhood theorem for compact Ricci flows, see \cite{P1}\cite{KL08}\cite{MT}\cite{CZ}\cite{B1}, etc.. The argument here is easier since we assume uniformly PIC and weakly PIC$_2$ conditions, which in particular implies the nonnegative Ricci curvature. We sketch the proof for the reader's convenience.

Assume that there exists an $\bar \ep>0$ such that the conclusion does not hold for a sequence $(x_i,t_i) \in M \times [1,2]$ with $Q_k=R(x_k,t_k) \to \infty$. By a standard point-picking argument, we can assume that for any $A>0$ and any $(y,t) \in B_{g(t_k)}(x_k,AQ_k^{-1/2}) \times [t_k-AQ_k^{-1},t_k]$ with $R(y,t) \ge 2Q_k$, the conclusion of the theorem holds. Indeed, otherwise for a fixed $k$ there exists a spacetime sequence $(y_i,s_i)$ with $(y_0,s_0)=(x_k,t_k)$ satisfying (1) the conclusion fails, (2) $R_{i+1} \ge 2R_i$ and (3) $(y_{i+1},s_{i+1}) \in B_{g(s_i)}(y_i,AR_i^{-1/2}) \times [s_i-AR_i^{-1},s_i]$. Here we denote $R(y_i,s_i)$ by $R_i$. By our assumption, the distance is nonincreasing for $t$. Therefore, we have
\begin{align*}
&d_{g(t_k)} (x_k,y_I) \le \sum_{i=0}^{I-1} d_{g(s_i)}(y_i,y_{i+1}) \le  \sum_{i=0}^{I-1} AR_i^{-1/2} \le 4AR_0^{-1/2}, \\
&0\le t_k-s_I \le \sum_{i=0}^{I-1}(s_i-s_{i+1}) \le \sum_{i=0}^{I-1} AR_i^{-1} \le  2AR_0^{-1/2}.
\end{align*}
Therefore, the process must end after finite steps, which is a contradiction. By taking a diagonal sequence, we consider the spacetime limit of $(M,g_k(t),x_k)$ for $t \le 0$, where $g_k(t)=Q_kg(Q_k^{-1}t+t_k)$. To derive a contradiction, we only need to prove the limit is a $\kappa$-solution.

\textbf{Step 1}: We claim there exists a sequence $H_k \to \infty$ and constants $\eta_m>0,c>0$ satisfying the following. For any $(\bar x,\bar t) \in B_{g(t_k)}(x_k, H_k Q_k^{-1/2}) \times [t_k-H_kQ_k^{-1},t_k]$, if we set $\bar Q=Q_k+R(\bar x,\bar t)$, then on the parabolic neighborhood $P=B_{g(\bar t)}(\bar x,c\bar Q^{-1/2}) \times [\bar t-c\bar Q^{-1},\bar t]$ we have
\begin{align}
|\na^m R| \le \eta_m \bar Q^{\frac{m}{2}+1} \label{eq:501a}
\end{align}
for $m \ge 0$. Indeed, if $\bar Q \ge 3 Q_k$, then by our assumption $(\bar x,\bar t)$ has a canonical neighborhood and hence $|\partial_t R^{-1}|+|\na R^{-1/2}| \le C$. Therefore, the local geometry around $(\bar x, \bar t)$ is well-controlled. Moreover, the higher curvature estimates follow from Shi's local estimates. For more details, see \cite[Theorem $4.1$, Step $1$]{CZ} or \cite[Lemma $52.11$]{KL08}.

\textbf{Step 2}: Next, we prove that $(M,g_k(0),x_k)$ converges smoothly to a complete smooth Riemannian manifold $(M_{\infty},g_{\infty},x_{\infty})$. From Step $1$, there exist constants $c_1>0,C_1>0$ such that $R_{g_k(0)} \le C_1$ on $B_{g_k(0)}(x_k,c_1)$. Therefore, it follows from the $\kappa$-noncollapsing condition that $\text{Vol}\,B_{g_k(0)}(x_k,1) \ge v_0>0$. From the standard Bishop-Gromov volume comparison theorem, for any $L>0$ and $y \in B_{g_k(0)}(x_k,L)$, we have $\text{Vol}\, B_{g_k(0)}(y,1) \ge v_1$, where $v_1=v_1(v_0,n,L)>0$. Therefore, it follows from Lemma \ref{lem:pe} that there exist $C_2=C_2(v_1)>0$ and $\tau=\tau(v_1)>0$ such that
\begin{align}
R(x,t) \le C_2  \label{eq:501b}
\end{align}
for $(x,t) \in B_{g_{k}(0)}(y,1/4) \times [-\tau,0]$. Combining \eqref{eq:501a} and \eqref{eq:501b}, one easily concludes that the limit of $(M,g_k(0),x_k)$ is a complete smooth Riemannian manifold, which has uniformly PIC$_1$, weakly PIC$_2$, and is $\kappa$-noncollapsed.

\textbf{Step 3}: Next, we show that the curvature of the limit $(M_{\infty},g_{\infty},x_{\infty})$ must be uniformly bounded. From the proof of Step $2$, we can assume $(M_{\infty},g_{\infty},x_{\infty})$ is defined on a spacetime open neighborhood of $M_{\infty} \times (-\infty,0]$ which contains $M_{\infty} \times \{0\}$. If the curvature operator of $g_{\infty}(0)$ lies on the boundary of PIC$_2$ at some point, then it follows from \cite[Proposition 6.6]{B1} that the universal covering $( \tilde M_{\infty},g_{\infty}(0))$ is isometric to $(N \times \R,g_1 \times g_E)$, where $(N, g_1)$ is a complete Riemannian manifold with uniformly PIC$_1$ and weakly PIC$_2$. We claim that $N$ has bounded curvature. Otherwise there exists a sequence $q_k \in N$ such that $R_{g_1}(q_k) \to \infty$. By our assumption of the canonical neighborhood, $(N,g_1,q_k)$ is $2\bar \ep$-close to the standard $S^{n-1}/\Gamma$ and hence $N$ is compact, which is a contradiction.

Therefore, we may assume $(M_{\infty},g_{\infty})$ has strictly PIC$_2$. If the curvature is not bounded, we can choose $z_k \in M_{\infty}$ such that $R_{g_{\infty}}(z_k) \to \infty$. Similarly, we assume that $(M_{\infty},g_{\infty},z_k)$ is $2\bar \ep$-close to a $\kappa$-solution. By choosing a different sequence of points if necessary, we assume all $(M_{\infty},g_{\infty},z_k)$ is $2\bar \ep$-close to the standard $(S^{n-1}/\Gamma) \times \R$. In addition, it follows from \cite[Proposition A.2]{B2} that $\Gamma=\{1\}$. However, it contradicts \cite[Proposition $2.2$]{CZ}.

\textbf{Step 4}: Now $(M_{\infty},g_{\infty}(t))$ can be extended backwards to an ancient solution with uniformly bounded curvature. The proof of the claim follows verbatim from \cite[Theorem $4.1$, Step $4$]{CZ}, see also \cite[Step 4 in Page 2705]{KL08} and \cite[Theorem $7.2$, Step 6]{B1}.

In sum, we have proved that $(M_{\infty},g_{\infty}(t))_{t\in (-\infty,0]}$ is a $\kappa$-solution, which contradicts our assumptions on $x_k$.
\end{proof}

Next, we prove the following lemma, see also \cite[Theorem $3.6$]{C}.

\begin{lem} \label{lem:501}
Let $(M^n,g(t))_{t \in [0,T]}$ be a complete, $\kappa$-noncollapsed solution of the Ricci flow with weakly \emph{PIC}$_2$ and $(M,g(t))$ has bounded curvature for each $t \in [0,T]$. Then $(M,g(t))_{t \in [0,T]}$ has uniformly bounded curvature. 
\end{lem}

\begin{proof}
We claim that for any $t_0 \in [0,T]$, there exists an $\ep>0$ such that $R$ is uniformly bounded on $M \times [t_0,t_0+\ep)$. Indeed, since $(M,g(t_0))$ has bounded curvature and is $\kappa$-noncollapsed, it clear that there exists a $v_0>0$ such that $\text{Vol}\,B_{g(t_0)}(x,1) \ge v_0$ for any $x \in M$. Therefore by \cite[Corollary $1.3$]{CW}, there exists an $\ep>0$ such that $R \le C(t-t_0)^{-1}$ for $t \in [t_0,t_0+\ep)$. Now the claim follows from \cite[Theorem 3.1]{C}.

Next we define $I:= \{t \in [0,T] \mid \sup_{M \times [0,t]}R<\infty\}$. It is clear from the claim that $I$ is open and nonempty. On the other hand, for any $t_i \in I$ such that $\lim_{i \to \infty} t_i=\bar t$, we know from our definition that $(M,g(t))_{t \in [0, \bar t)}$ has bounded curvature in any compact time interval. Therefore, the trace Harnack inequality holds, see \cite{B}. In particular, $R$ is uniformly bounded on $M \times [0, \bar t]$, since 
\begin{align*}
tR(x,t) \le \bar t \lc \sup_{M \times \{\bar t\}} R \rc
\end{align*}
for any $x \in M$ and $t \le \bar t$. Since $I$ is both open and closed, $T \in I$ and the proof is complete.
\end{proof}

Now, we prove the main result of the section, which states that the assumption of the trace Harnack inequality in \cite[Proposition $6.11$]{B1} can be dropped.

\begin{prop} \label{prop:bdd}
Let $(M^n,g(t))_{t \in (-\infty,0]}$ be an $n$-dimensional, $\kappa$-noncollapsed, noncompact complete ancient solution to the Ricci flow with uniformly $\mathrm{PIC}$ and weakly $\mathrm{PIC}_2$. Then the curvature of $(M,g(t))_{t \in (-\infty,0]}$ is uniformly bounded. In particular, under the same assumption, the trace Harnack inequality holds for $(M,g(t))_{t \in (-\infty,0]}$.
\end{prop}
\begin{proof}
If there exists $(x_0,t_0) \in M \times (-\infty,0]$ such that the curvature operator $\Rm(x_0,t_0)$ lies on the boundary of PIC$_2$ cone, then the universal cover $\tilde{M}$ of $M$ splits off a line by \cite[Proposition 6.6]{B1}. We assume $\tilde{M}$ is isometric to $N \times \R$ where $(N,g_1(t))_{t \in (-\infty,0]}$ is $(n-1)$-dimensional, $\kappa$-noncollapsed, complete ancient solution to the Ricci flow with uniformly PIC$_1$, see \cite[Proposition $7.14$]{RFST}. Therefore, it follows from \cite[Theorem 6.4]{B1} (see also \cite{Yo17}) that $N$ is homothetic to $S^{n-1}$. In this case, the conclusion is obviously true.

Therefore, we may assume that $M$ has strictly $\mathrm{PIC}_2$. From the Lemma \ref{lem:501}, we only need to prove the curvature is bounded for each time slice. Fix a $t_0 \le 0$, if the curvature at $t_0$ is unbounded, there exists a sequence $p_i$ with $Q_i=R(p_i,t_0) \to \infty$. By applying Theorem \ref{thm:cano} on $M \times [t_0-2,t_0]$, we conclude that $(M,Q_ig(t_0),p_i)$ converges smoothly to a $\kappa$-solution. Since the limit must contain a splitting direction, by Toponogov's splitting theorem, we may assume the limit is the standard $(S^{n-1}/\Gamma)\times \R$. By our assumption, it follows from \cite[Proposition A.2]{B2} that $\Gamma=\{1\}$. However, we obtain a contradiction by \cite[Proposition $2.2$]{CZ}.

By Lemma \ref{lem:501}, $(M^n,g(t))_{t \in (-\infty,0]}$ has bounded curvature on each compact time interval and hence the trace Harnack inequality holds, see \cite{B}. Therefore, the curvature is uniformly bounded since $R$ is nondecreasing along $t$.
\end{proof}

By combining Theorem \ref{thm:kappa} and Proposition \ref{prop:bdd}, we obtain the following theorem.

\begin{thm}\label{thm:kappa1}
Suppose $(M^n,g(t))_{t  \in (-\infty,0]}$ is a complete, noncompact, $\kappa$-noncollapsed ancient solution to the Ricci flow with uniformly \emph{PIC} and weakly \emph{PIC}$_2$. Then it is isometric to either a family of shrinking cylinders (or a quotient thereof ) or to the Bryant soliton.
\end{thm}
\begin{rem}
The conclusion of Theorem \ref{thm:kappa1} also holds for any $3$-dimensional, $\kappa$-noncollapsed, noncompact complete ancient solution to the Ricci flow, based on Brendle's breakthrough \emph{\cite[Theorem $1.3$]{B0}}. On the other hand, it is not clear if there is any $3$-dimensional noncompact complete ancient solution to the Ricci flow which has unbounded curvature. For immortal solutions with nonnegative curvature operator and unbounded curvature, the readers can refer to \emph{\cite[Theorem $1.4$]{CW}}.
\end{rem}

\emph{Proof of Theorem \ref{thm:main1}}: Theorem \ref{thm:main1} follows immediately from Theorem \ref{thm:kappa1} and the curvature improvements obtained from Theorem \ref{thm:pic2} and Lemma \ref{lem:403}.

\section{Proof of Theorem \ref{thm:main2}}

\subsection{$\kappa$-solutions to the K\"ahler Ricci flow}

We first recall the following definition in the \ka setting.

\begin{defn}\label{def:kappa1} \emph{[}$\kappa$-solutions to the \ka Ricci flow\emph{]}
A complete ancient solution $(M^n,g(t))_{t  \in (-\infty,0]}$ to the \ka Ricci flow is called a $\kappa$-solution if it satisfies (1) $\kappa$-noncollapsed, (2) nonnegative bisectional curvature, and (3) uniformly bounded curvature.
\end{defn}

Notice that all $\kappa$-solutions (with a fixed $\kappa$) are compact in the $C^{\infty}_{\text{loc}}$-topology, see \cite[Theorem $2.1$]{N}. Next, we prove that any $\kappa$-solution must be of Type I, see also \cite[Lemma $2.5$]{DZ1}.

\begin{lem} \label{lem:type2}
For any $\kappa$-solution $(M^n,g(t))_{t  \in (-\infty,0)}$ such that $t=0$ is the singular time, there exists a constant $C_0>0$ such that $|t||R(x,t)| \le C_0$ for all $(x,t)  \in M\times (-\infty,0)$.
\end{lem}
\begin{proof} 
Suppose that it is not true. We take $T_i \to -\infty$ and $\epsilon_i \to 0^-$ and choose $(x_i,t_i) \in M \times (T_i,\epsilon_i)$ so that 
$$(\epsilon_i-t_i)(t_i-T_i)R(x_i,t_i)=(1-\delta_i)\sup_{M \times [T_i,\epsilon_i]} (\epsilon_i-t)(t-T_i)R(x,t),$$
where $\delta_i \to 0$. Then it follows from direct computation, see \cite[Proposition $8.20$]{CLN06} for details, that for the rescaled metric $g_i(t):=Q_i g(t_i+Q_i^{-1}t)$ where $Q_i:=R(x_i,t_i)$, $(M,g_i(t),x_i)$ converges smoothly to $(M_\infty,g_\infty(t),x_\infty)_{t \in (-\infty,\infty)}$ which is a nonflat $\kappa$-noncollapsed eternal solution to the \ka Ricci flow. In particular, this limit has nonnegative bisectional curvature. Moreover, $R_{g_{\infty}}(x,t)\le 1$ and $R_{g_{\infty}}(x_\infty,0)= 1$ by our construction. After considering the universal cover of $M_\infty$ and Cao's dimension reduction argument \cite[Theorem $2.1$]{Ca1}, we may assume that $M_\infty$ is simply connected and has positive Ricci curvature. By using the result in \cite[Theorem $1.3$]{Ca2}, we conclude that $(M_\infty,g_\infty)$ is a nonflat, $\kappa$-noncollapsed \ka Ricci steady soliton with nonnegative bisectional curvature. However, such a \ka steady soliton does not exist by \cite[Theorem $1.2$]{DZ}.
\end{proof}

Notice that the \ka Ricci shrinker with nonnegative bisectional curvature is an important type of $\kappa$-solution. We recall its classification.

\begin{thm}[Theorem 3 of \cite{N}] \label{thm:shrinker}
Let $(M^n,g,f)$ be a \ka Ricci shrinker with nonnegative bisectional curvature. Then $(M,g)$ is isometric-biholomorphic to a quotient of $N^k \times \C^{n-k}$, where $N$ is a compact Hermitian symmetric spaces.
\end{thm}

Next, we recall that all compact $\kappa$-solutions are completely classified.

\begin{thm}[Proposition 2.8 of \cite{DZ1}] \label{thm:cpt}
Let $(M,g(t))_{t \in (-\infty,0]}$ be a compact $\kappa$-solution. Then it must be isometric-biholomorphic to a quotient of compact Hermitian symmetric space.
\end{thm}

In particular, if we further assume $M$ to be complex 2-dimensional, then it must be isometric-biholomorphic to $\CP^2$ or $\CP^1 \times \CP^1$, up to scalings on each factor. From now on, we will consider the noncompact case. First, we prove the following elementary result.

 \begin{lem} \label{lem:quo}
 Let $(M^2,g(t))$ be a $\kappa$-solution whose universal cover is $\C \times \CP^1$. Then it is biholomorphic-isometric to a family of shrinking $\C \times \CP^1$ itself.
 \end{lem}
 \begin{proof} First, we investigate an isometry of $\C \times \CP^1$. Let $\Phi : \C \times \CP^1 \to \C \times \CP^1$ and $v_1$ is a vector field tangent to $\C$. Then $\Phi_*(v_1)$ is also tangent to $\C$ since it is parallel. Therefore, $\Phi$ preserves the product structure and we can write $\Phi=(\Phi_1,\Phi_2)$ where $\Phi_1 : \C \to \C$ and $\Phi_2 : \CP^1 \to \CP^1$ are isometries. Now, let $f : \C \times \CP^1 \to \C \times \CP^1$ be an isometry corresponding to the covering map $\C \times \CP^1 \to M$. If $f$ is not trivial, we can write $f=(f_1,f_2)$ where $f_1$ is an orientation preserving rigid motion on $\R^2$ and $f_2 \in \text{SO}(3)$. In particular, we know that $f_2 \in \text{SO}(3)$ has a fixed point. Therefore, $f_1$ has no fixed point and hence generates a infinite group. Now we show that $(M,g(t))_{t \in (-\infty,0)}$ is not $\kappa$-noncollapsed. By our assumption, $|\Rm(x,-r^2)| \le Cr^{-2}$. For a fixed $\bar x$, it is easy to see
$$
\lim_{r \to \infty}\frac{\text{Vol}\,B_{g(-r^2)}(\bar x,r)}{r^4}=0,
$$
which is a contradiction.
\end{proof}
For later applications, we recall the following theorem of Type I ancient solution, see also \cite{XCQZ}.

\begin{thm} \label{thm:asm}
Let $(M,g(-\tau))_{\tau \in (0,\infty)}$ be a $\kappa$-solution and $R(x,-\tau) < C_0/\tau$ for any $x \in M$ and $\tau \in (0,\infty)$. Then we have the following.
\begin{enumerate}[label=(\roman*)]
    \item For any $0<\tau_1<\tau_2$ and $p,q \in M$, we have
$$d_{g(-\tau_2)}(p,q)-8(n-1)C_0(\sqrt{\tau_2}-\sqrt{\tau_1})\le d_{g(-\tau_1)}(p,q) \le d_{g(-\tau_2)}(p,q).$$
\item For any $p \in M$ and $\tau_i \to \infty$, if we set $g_i(-\tau)=\tau_i^{-1}g(-\tau_i\tau)$, then the sequence $(M,g_i(-\tau),p)_{\tau \in (0,\infty)}$ subconverges smoothly to a complete, nonflat \ka Ricci shrinker $(M_\infty,g_\infty(-\tau),p_\infty)_{\tau \in (0,\infty)}$ with nonnegative bisectional curvature.
\end{enumerate}
\end{thm}
\begin{proof}
(i) follows directly from \cite[Lemma $8.3$]{P} and (ii) follows from \cite[Theorem $3.1$]{AN} by using the reduced distance introduced in \cite{N}. 
\end{proof}

Next, we prove the following characterization of the behavior near the singular time. Here, a \ka manifold $M$ is said to be irreducible if its universal cover is not isometric-biholomorphic to a product of two \ka manifolds of smaller dimensions. From the uniqueness of the Ricci flow \cite{CZ06} and \cite[Corollary $1.2$]{Kot14}, we know that in our case a \ka Ricci flow is irreducible if and only if any time slice is irreducible.

\begin{prop} \label{prop:her}
Let $(M,g(-\tau))_{\tau \in (0,\infty)}$ be an irreducible $\kappa$-solution and let $\tau=0$ be the singular time. Then the followings are equivalent.
\begin{enumerate}[label=(\alph*)]
    \item $M$ is isometric-biholomorphic to a quotient of compact Hermitian symmetric space.
    \item There exists $p \in M$ such that $\lim_{\tau \to 0}R(p,-\tau)=\infty$.
    \item For all $p \in M$, we have $\lim_{\tau \to 0}R(p,-\tau)=\infty$.
\end{enumerate}
\end{prop}
\begin{proof}
The implication (a) $\Rightarrow$ (b) is obvious.

(b) $\Rightarrow$ (c) :Suppose not. Then there exists $p,q \in M$ such that for $\tau_i \to 0$, we have $R(p,-\tau_i) \to \infty$ but $R(q,-\tau_i) \le C$ for some $C>0$. Since $d_{g(-\tau_i)}(q,p) \le d_{g(-\tau_1)}(q,p)$ is uniformly bounded, it follows from the compactness of $\kappa$-solutions (see \cite{P} \cite{N}) that $R(p,-\tau_i) \le CR(q,-\tau_i) \le C_1$. This is a contradiction.

(c) $\Rightarrow$ (a) : Fix $\tau>0$. Then for any $\tau_0\in(0,\tau)$ and $p \in M$, from the estimate $\left| \frac{\del R}{\del \tau} \right| \le \eta R^2$ we obtain
$$
\frac{1}{R(p,-\tau)}-\frac{1}{R(p,-\tau_0)} \le \eta(\tau-\tau_0).
$$
Now, by taking $\tau_0 \to 0$, since $R(p,-\tau_0) \to \infty$, we have $R(p,-\tau) \ge \frac{1}{\eta\tau}$. Therefore, the scalar curvature of $(M,g(-\tau))$ has a positive lower bound. If $M$ is noncompact, it contradicts the average curvature decay in \cite[Theorem $0.4$]{NT}. So $M$ has to be compact. Now the result follows from Theorem \ref{thm:cpt}
\end{proof}

Next, we prove the following result about the asymptotic scalar curvature ratio.
\begin{lem} \label{lem:ascr}
Let $(M,g(-\tau))_{\tau \in (0,\infty)}$ be a $\kappa$-solution. Then for any $\tau \in (0,\infty)$ and $p \in M$, we have 
$$
\liminf_{d_{g(-\tau)}(x,p)\to \infty}R(x,-\tau)d^2_{g(-\tau)}(x,p)=\infty.
$$
\end{lem}
\begin{proof} Suppose not. Then we may assume it does not hold when $\tau=1$. So there exists a sequence $\{p_i\} \subset M$ such that $d_{g(-1)}(p,p_i) \to \infty$ but $R(p_i,-1)d^2_{g(-1)}(p,p_i) \le C$ for some $C>0$. If we take $\rho_i=d^2_{g(-1)}(p,p_i)$ and consider $g_i(-\tau)=\rho_i^{-1}g(\rho_i(1-\tau)-1)$, then from the compactness of $\kappa$-solutions, a sequence $(M,g_i(-\tau),p_i)_{\tau \in [1,\infty)}$ subsequentially converges to $(M_\infty,g_\infty(-\tau),p_\infty)_{\tau \in [1,\infty)}$. Moreover, since $\rho_i \to \infty$, one can use Theorem \ref{thm:asm} to show that $(M,g_i(-\tau),p)_{\tau \in (1,\infty)}$ subsequentially converges to a nonflat \ka Ricci shrinker $(M'_\infty,g'_\infty(1-\tau),p'_\infty)_{\tau \in (1,\infty)}$. But since $d_{g_i(-1)}(p,p_i)=1$ for all $i$, we know that both limits are isometric. Note that $$R_\infty(p_\infty,-1)=\lim_{i \to \infty}R_i(p_i,-1)=\lim_{i \to \infty}d^2_{g(-1)}(p,p_i)R(p_i,-1)\le C$$
from the assumption. However, since $\tau=1$ is the singular time of the \ka Ricci shrinker, we have $R_{g_\infty}(q,-\tau) \to \infty$ for any $q \in M_\infty$  as $\tau \to 1$, which is a contradiction. 
\end{proof}

Using this result, we can investigate the asymptotic behavior of the ancient solution in the following way.
\begin{prop} \label{prop:beh}
Let $(M^2,g(-\tau))_{\tau \in [0,\infty)}$ be a noncompact complex 2-dimensional nonflat $\kappa$-solution. Then for any sequence $(p_i,\tau_i) \in M \times (0,\infty)$ with $\tau_i \to \infty$ and $Q_i=R(p_i,-\tau_i)$, the sequence $(M,Q_ig(-\tau_i),p_i)$ subsequentially converges to $\C \times \CP^1$ in the Cheeger-Gromov sense. Here we assume the scalar curvature of $\C \times \CP^1$ is identically $1$.
\end{prop}
\begin{proof} 
First, we consider the case when $p_i=p$ for all $i$. Let $g_i(-\tau)=\tau_i^{-1}g(-\tau_i\tau)$. From Theorem \ref{thm:asm}, we know that the limit of $(M,g_i(-\tau),p)_{\tau \in (0,\infty)}$ exists as a nonflat, noncompact, K\"{a}hler Ricci shrinker with nonnegative bisectional curvature. From Theorem \ref{thm:shrinker} and Lemma \ref{lem:quo}, we know that this limit has to be a shrinking $(\C \times \CP^1,g_\infty(-\tau))_{\tau \in (0,\infty)}$. We assume that $R(p,-\tau_i)\tau_i \to L$, then the limit of $(M,g_i(-1),p)$ is homothetic to the limit of $(M,\tilde{g}_i(-1),p)$ where $\tilde{g}_i(-1)=Q_ig(-\tau_i)$. Therefore, the conclusion holds. Now we assume that $R(p,-\tau_i)\tau_i\to 0$. From $\left| \frac{\del R}{\del \tau} \right| \le \eta R^2$, we have
$\frac{1}{R(p,-\tau_i)}-\frac{1}{R(p,0)} \le \eta\tau_i$. This inequality implies that
$$
1-\eta\tau_iR(p,-\tau_i) \le \frac{R(p,-\tau_i)}{R(p,0)} \le 1.
$$
So we get $\frac{R(p,-\tau_i)}{R(p,0)} \to 1$ as $i \to \infty$. But then the trace Harnack inequality implies that $R(p,-\tau)$ is constant in $\tau$. From the same argument as in the proof of Lemma \ref{lem:type2}, we conclude that $M$ is flat, which is a contradiction. Therefore, it completes the proof when the sequence $\{p_i\}$ is a fixed point.

If there exists a point $p \in M$ such that $\frac{d_{g(0)}(p,p_i)}{\sqrt{\tau_i}} \le D$ for all $i$, then the limits of $(M,g_i(-1),p)$ and $(M,g_i(-1),p_i)$ are isometric to each other since by Theorem \ref{thm:asm}
$$
d_{g_i(-1)}(p,p_i)=\tau_i^{-\frac12}d_{g(-\tau_i)}(p,p_i)\le \tau_i^{-\frac12}(D\sqrt{\tau_i}+8C_0\sqrt{\tau_i})= 8C_0+D
$$
for large enough $i$. Therefore, we may assume $\frac{d_{g(0)}(p,p_i)}{\sqrt{\tau_i}} \to \infty$. There are two possibilities. 

\textbf{Case 1.} $Q_i\tau_i \to 0$

Let $\tilde{Q}_i=R(p_i,0)$. By integrating $\left| \frac{\del R}{\del \tau} \right| \le \eta R^2$ and multiplying $Q_i$ on both sides, we get
$$0\le 1-\frac{Q_i}{\tilde{Q}_i} \le \eta Q_i\tau_i \to 0.$$
So we have $\frac{Q_i}{\tilde{Q}_i} \to 1$ as $i \to \infty$. In particular, we have $\tilde{Q}_i\to 0$. It implies that the limit, denoted by $(M_{\infty},g_{\infty}(-\tau),p_{\infty})_{\tau \in (0,\infty)}$, of $(M,Q_ig(-Q_i^{-1}\tau),p_i)_{\tau \in (0,\infty)}$ is isometric to that of $(M,\tilde{Q}_ig(-\tilde{Q}_i^{-1}\tau),p_i)_{\tau \in (0,\infty)}$. Note that we know
$\tilde{Q}_id^2_{g(0)}(p,p_i)\to \infty$
from Lemma \ref{lem:ascr} and $\tilde{Q}_ig(-\tilde{Q}_i^{-1}\tau)$ has a nonnegative curvature operator from Lemma \ref{lem:405}, we can apply Toponogov's splitting theorem to conclude that the limit of $(M,\tilde{Q}_ig(-\tilde{Q}_i^{-1}\tau),p_i)$ splits a line generated by the vector field $V$. Since the limit is equipped with a complex structure $J_{\infty}$, the vector field $J_{\infty}V$ generates another splitting direction. Therefore, the universal covering $(\tilde M_{\infty}, \tilde g_{\infty}(\tau),p_{\infty})$ is biholomorphic-isometric to $\C \times \CP^1$, where we have used the fact that every real 2-dimensional nonflat $\kappa$-noncollapsed ancient solution is a shrinking sphere \cite[Corollary $11.3$]{P}. Therefore, it follows from Lemma \ref{lem:quo} that $(M_{\infty},g_{\infty}(-\tau),p_{\infty})_{\tau \in (0,\infty)}$ is biholomorphic-isometric to a family of shrinking $(\C \times \CP^1,g_\infty(-\tau))_{\tau \in [0,\infty)}$ with a unit scalar curvature at $\tau=1$. Therefore, $(M,Q_ig(-\tau_i),p_i)$ converges smoothly to $\C \times \CP^1$.

\textbf{Case 2.} $Q_i\tau_i \to C>0$

For simplicity, we assume that $C=1$. Then there are 2 possible subcases. If $\tilde{Q_i}\tau_i \to L$ for some finite $L>0$, then it implies that the limit of $(M,Q_ig(-Q_i^{-1}\tau),p_i)_{\tau \in (0,\infty)}$ is homothetic to that of $(M,\tilde{Q}_ig(-\tilde{Q}^{-1}_i\tau),p_i)_{\tau \in (0,\infty)}$. From the previous argument, we know that the limit of $(M,\tilde{Q}_ig(-\tilde{Q}^{-1}_i\tau),p_i)_{\tau \in (0,\infty)}$ is isometric to $(\C \times \CP^1,g_\infty(-\tau))_{\tau \in (0,\infty)}$. Therefore, it is clear that the limit of $(M,Q_ig(-\tau_i),p_i)$ is also $(\C \times \CP^1,g_\infty(-1))$.

Now we assume that $\tilde{Q}_i\tau_i \to \infty$. From the compactness, for $\hat{g}_i:=\tau_i^{-1}g(-\tau_i\tau)$, a sequence $(M,\hat{g}_i(-\tau),p_i)_{\tau \in (0,\infty)}$ subsequentially converges to a noncompact $\kappa$-solution $(M_\infty,\hat{g}_\infty(-\tau),p_\infty)_{\tau \in (0,\infty)}$. If $M_\infty$ is reducible, then we can argue as above that the limit is $\C \times \CP^1$. In this case, since we already have $Q_i\tau_i \to 1$, the limit of $(M,\hat{g}_i(-\tau),p_i)_{\tau \in (0,\infty)}$ is isometric to the limit of $(M,Q_ig(-\tau_i\tau),p_i)_{\tau \in (0,\infty)}$ which verifies the statement by taking $\tau=1$. Therefore, we may assume that $M_\infty$ is irreducible. By integrating $\left| \frac{\del}{\del \tau} R \right| \le \eta R^2$ from 0 to $-\tau_i\tau$ and dividing both sides by $\tau_i$, we obtain
$$\frac{1}{\tau_iR(p_i,-\tau_i\tau)}-\frac{1}{\tilde{Q}_i\tau_i}\le \eta \tau$$
Since $\tilde{Q}_i\tau_i \to \infty$ from the assumption, it implies that $\hat{R}_i(p_i,-\tau)=\tau_iR(p_i,-\tau_i\tau)\ge \frac{1}{2\eta \tau}$ for large enough $i$. In particular, we can deduce that $\lim_{\tau \to 0}\hat{R}_\infty(p_\infty,-\tau)=\infty$. From Proposition \ref{prop:her}, it implies that $M_\infty$ has to be compact, which is a contradiction.
\end{proof}

\subsection{Construction of the fibration}

With the help of Proposition \ref{prop:beh}, the following proposition is immediate. In the following, we will drop the time parameter $-\tau$ for simplicity unless there is confusion. For all the spheres $S^2$, we assume the scalar curvature is identically $1$.

\begin{prop} \label{prop:loc}
Let $(M^2,g(-\tau))_{\tau \in [0,\infty)}$ be a noncompact complex 2-dimensional nonflat $\kappa$-solution. For any $\ep>0$, there exists a $\bar \tau>0$ such that for any $(x,-\tau) \in M \times (-\infty, -\bar \tau]$, there exists an open neighborhood $\Omega_x \ni x$ with a diffeomorphism $\psi_x : \Omega_x \to B(0,100) \times S^2 \subset \R^2 \times S^2$ such that 
\begin{enumerate}[label=(\alph*)]
\item  $\psi_x(x)=(0, \bar s)$, where $\bar s$ is the north pole of $S^2$.
\item For the standard metric $g_0$ on $\R^2 \times S^2$ and $i \in [0,\epsilon^{-1}]$ and $g_x=R(x)g$, we have $$\sup_{\Omega_x}|\na_{g_x}^i(g_x-\psi_x^*g_0)|_{g_x}\le \epsilon.$$
\item The map $\varphi_x=\pi_1 \circ \psi_x : (\Omega_x,g_x) \to (B(0,99),g_0|_{B(0,99)})$ is an $\ep$-Riemannian submersion.
\end{enumerate}
\end{prop}

Now, we will construct a transition map $\varphi_{x,y}$ between two local fibrations $\varphi_x$ and $\varphi_y$ defined in Proposition \ref{prop:loc}. This strategy originates in \cite{CG,CFG}, see also \cite{CL}.  In the following, the function $\delta(\epsilon) \to 0$ as $\epsilon \to 0$ and $\delta(\ep)$ may be different line by line.

\begin{prop} \label{prop:pat0}
With the same assumptions as in Proposition \ref{prop:loc}, for any $x,y \in M$ with $d_{g}(x,y)\le 10 r$ with $r=r(x,y)=1/\sqrt{\max(R(x),R(y))}$, if we set $\Omega_{x,y}=\Omega_{x} \cap \Omega_{y}$, then there is a $\delta(\epsilon)$-almost isometry $\varphi_{x,y} : \varphi_{y}(\Omega_{x,y}) \to \varphi_{x}(\Omega_{x,y})$. Moreover, it satisfies the following properties.
\begin{enumerate}[label=(\alph*)]
    \item $|\varphi_x-\varphi_{x,y}\circ \varphi_y| \le \delta(\epsilon) r$.
    \item $|D\varphi_x-D\varphi_{x,y}\circ D\varphi_y|\le \delta(\epsilon)$.
\end{enumerate}
\end{prop}
\begin{proof} 
From the assumption, we have $d_{g_y}(x,y)=\sqrt{R(y)}d_g(x,y)\le 10$ and hence $x \in \Omega_y$. It follows from Proposition \ref{prop:loc} that
\begin{align}
\left| \frac{R(x)}{R(y)}-1 \right| \le \delta(\epsilon). \label{eq:600}
\end{align}
Moreover, if we set $g_1=(\psi_x \circ \psi_y^{-1})^* g_0$, then it follows from Proposition \ref{prop:loc} and \eqref{eq:600} that $g_1$ and $g_0$ are $C^2$-close on $\psi_y(\Omega_{x,y})$. More precisely, on $\psi_y(\Omega_{x,y})$ one has
\begin{align}
|g_1-g_0|+|\na_{g_0}g_1|+|\na^2_{g_0} g_1| \le \delta(\ep), \label{eq:601}
\end{align}
where the norms are with respect to $g_0$. Next, we prove the map $\psi_x \circ \psi_y^{-1}$ almost preserves the product structure. Indeed, if $V$ is a parallel vector field along $\R^2$, with respect to $g_0$, and $V_1=(\psi_x \circ \psi_y^{-1})_* V$, then from \eqref{eq:601} we have
\begin{align}
|\na_{g_0} V_1| \le \delta(\ep), \label{eq:602}
\end{align}
and hence $V_1$ is almost tangent to $\R^2$ in the sense that
\begin{align}
| (\pi_1)_*V_1-V_1| \le \delta(\ep), \label{eq:603}
\end{align}
where $\pi_1 : \R^2 \times S^2 \to \R^2$ is the projection map. Similarly, if $V_2$ is tangent to $S^2$, then we have
\begin{align}
| (\pi_2)_*V_2-V_2| \le \delta(\ep). \label{eq:603a}
\end{align}

Now we define $\varphi_{x,y}:\varphi_y(\Omega_{x,y}) \to \varphi_x(\Omega_{x,y})$ by $\varphi_{x,y}(p)=\pi_1\circ \psi_x \circ \psi_y^{-1}(p, \bar s)$. It is clear from the definition that $\varphi_{x,y}$ is a $\delta(\epsilon)$-almost isomtery on $\varphi_y(\Omega_{x,y})$. We claim that $\varphi_{x,y}$ satisfies all required property. Indeed, for any $z \in \Omega_{x,y}$, we set $\psi_y(z)=(p_1,s_1)$ and $\psi_y(z)=(p_2,s_2)$. We consider a geodesic segment $\gamma$ such that $\gamma(0)=(p_1,s)$ and $\gamma(1)=(p_1,\bar s)$. If we set $\tilde \gamma=\psi_x \circ \psi_y^{-1} \circ \gamma$, then it is clear from \eqref{eq:603} that $|\pi_1 (\tilde \gamma(0))-\pi_1(\tilde \gamma(1))| \le \delta(\ep) r$. Therefore, the property (a) is proved by our definition of $\varphi_{x,y}$. The property (b) can be proved similarly by \eqref{eq:603} and \eqref{eq:603a}.
\end{proof}

\begin{rem}
To make the proof of Proposition \ref{prop:pat0} rigorous, one needs to slightly shrink $\Omega_{x},\Omega_{y}$ and $\Omega_{x,y}$ such that the new sets contain all $S^2$ fibers. In the following, we will not mention this explicitly.
\end{rem}

Next, we will show that the map $\varphi_{x,y}$ which is constructed previously almost satisfies the cocycle condition $\varphi_{x_1,x_3}=\varphi_{x_1,x_2}\circ \varphi_{x_2,x_3}$.
\begin{prop}
For any $x_1,x_2,x_3 \in M$ with $d_g(x_i,x_j) \le 10 r$ with \newline  $r=r(x_1,x_2,x_3)=1/\sqrt{\max(R(x_1),R(x_2),R(x_3))}$ for any pair $i,j \in \{1,2,3\}$, wherever it makes sense, we have the following.
\begin{enumerate}[label=(\alph*)]
    \item $|\varphi_{x_1,x_3}-\varphi_{x_1,x_2}\circ \varphi_{x_2,x_3}| \le  \delta(\epsilon) r.$
    \item $|D\varphi_{x_1,x_3}-D\varphi_{x_1,x_2}\circ D\varphi_{x_2,x_3}| \le \delta(\epsilon)$.
\end{enumerate}
\end{prop}
\begin{proof} 
From Proposition \ref{prop:pat0}, wherever it makes sense, we have
$$|\varphi_{x_1}-\varphi_{x_1,x_3}\circ \varphi_{x_3}|\le \delta(\epsilon)r~\text{and}~|\varphi_{x_2}-\varphi_{x_2,x_3}\circ \varphi_{x_3}|\le \delta(\epsilon)r.$$
Since $\varphi_{x_1,x_2}$ is $\delta(\epsilon)$-almost isometry, we have
\begin{align*}
|\varphi_{x_1}-\varphi_{x_1,x_2} \circ \varphi_{x_2,x_3} \circ \varphi_{x_3}|&\le |\varphi_{x_1}-\varphi_{x_1,x_2}\circ \varphi_{x_2}|+|\varphi_{x_1,x_2}\circ(\varphi_{x_2}-\varphi_{x_2,x_3}\circ \varphi_{x_3})| \le  \delta(\epsilon) r.
\end{align*}
Using the inequality $|\varphi_{x_1}-\varphi_{x_1,x_3}\circ \varphi_{x_3}|\le \delta(\epsilon) r$, we have 
$$
|(\varphi_{x_1,x_3}-\varphi_{x_1,x_2} \circ \varphi_{x_2,x_3})\circ \varphi_{x_3}| \le \delta(\epsilon) r.
$$
Since $\varphi_{x_3}$ is surjective, (a) is proved. Similarly, (b) can be proved by the same argument.
\end{proof}

Now we want to modify the local fibrations $\varphi_x$ to make them compatible with a transition map $\varphi_{x,y}$. To do so, we need the following lemma. 
\begin{lem} \label{lem:pat}
For any $x,y \in M$ and $r=r(x,y)=1/\sqrt{\max(R(x),R(y))}$ with $2r \le d_g(x,y) \le 4 r$, we assume that $B_g(x,2r)\cap B_g(y,2r) \ne \emptyset$. Then there exista a new fibrations $\tilde{\varphi}_x$ on $B_g(x,2r)$ such that $$\tilde{\varphi}_x=\varphi_{x,y}\circ \varphi_y$$ on $B_g(x,2r)\cap B_g(y,2r)$. Moreover, it has same estimates with those of $\varphi_x$ in Proposition \ref{prop:loc} and coincides with $\varphi_x$ on $B_g(y,4r(x,y))$ wherever $\varphi_x=\varphi_{x,y}\circ \varphi_y$.
\end{lem}
\begin{proof} 
Let $\theta(z)$ be a cut-off function on $\R^2$ such that $\theta(z)\equiv 1$ on $B(0,2)$ and $\theta(z)\equiv 0$ outside $B(0,4)$ and $\phi(p):=\theta\left(\frac{\varphi_y(p)}{r}\right)$. Now define a map $\tilde{\varphi}_x:\Omega_{x,y} \to \R^2$ by $\tilde{\varphi}_x(z)=\phi(z)(\varphi_{x,y}\circ \varphi_y(z))+(1-\phi(z))\varphi_x(z)$. Clearly, $\tilde{\varphi}_x=\varphi_x$ wherever $\varphi_x=\varphi_{x,y}\circ\varphi_y$. Also since $\phi(z)\equiv1$ for any $z \in B_g(x,2r) \cap B_g(y,2r)$, it satisfies the property. Now the estimates follow from $\tilde{\varphi}_x(z)-\varphi_x(z)=\phi(z)(\varphi_{x,y}\circ \varphi_y(z)-\varphi_x(z))$ and the Proposition \ref{prop:loc}. 
\end{proof}

\begin{lem} \label{lem:pat1}
For any $x_1,x_2,x_3 \in M$ and $r=r(x_1,x_2,x_3)=1/\sqrt{\max(R(x_1),R(x_2),R(x_3))}$ with $2r \le d_g(x_i,x_j) \le 4 r$ for any pair $i,j \in \{1,2,3\}$, we assume that $B_g(x_1,2r)$, $B_g(x_2,2r)$ and $B_g(x_3,2r)$ have a nonempty intersection. Then there exist a new diffeomorphism $\tilde{\varphi}_{x_1,x_3}$ on $\varphi_{x_3}(B_g(x_1,2r)\cap B_g(x_3,2r))$ such that $$\tilde{\varphi}_{x_1,x_3}=\varphi_{x_1,x_2}\circ \varphi_{x_2,x_3}$$ on $\varphi_{x_3}(B_g(x_1,2r)\cap B_g(x_2,2r) \cap B_g(x_3,2r))$. Moreover, it has the same estimates with those of $\varphi_{x_1,x_3}$ in Proposition \ref{prop:pat0} and coincides with $\varphi_{x_1,x_3}$ wherever $\varphi_{x_1,x_3}=\varphi_{x_1,x_2}\circ \varphi_{x_2,x_3}$.
\end{lem}
\begin{proof} 
Let $\theta(z)$ be the cut-off function defined in Lemma \ref{lem:pat}, then we define $\tilde{\varphi}_{x_1,x_3}(z)=\phi(z)(\varphi_{x_1,x_2}\circ\varphi_{x_2,x_3}(z))+(1-\phi(z))\varphi_{x_1,x_3}(z)$. Then one can check that $\tilde{\varphi}_{x_1,x_3}$ satisfies all properties mentioned in the statement.
\end{proof}

Now we can construct a global fibration on $M$.

\begin{prop} \label{prop:glob}
Let $(M^2,g(t))_{t\in (-\infty,0]}$ be a noncompact complex 2-dimensional nonflat $\kappa$-solution. Then there exists a smooth $S^2$-fibration $p : M \to \mathcal S$ where $\mathcal S$ is a noncompact Riemann surface.
\end{prop}

\begin{proof}
From Proposition \ref{prop:loc}, there exist local fibrations for all points on $M$, if $-t$ is sufficiently large. By using Lemma \ref{lem:pat} and Lemma \ref{lem:pat1}. we can follow the standard technique in \cite{CG,CFG} to modify all local fibrations to be compatible. For details, the reader can refer to \cite[Theorem $5.16$]{CL}.
\end{proof}

Now we can prove the following classification of the noncompact $\kappa$-solutions on \ka surface.
\begin{thm} \label{thm:kappa3}
Let $(M^2,g(t))_{t \in (-\infty,0]}$ be a nonflat, noncompact $\kappa$-solution. Then it is biholomorphic-isometric to a family of shrinking $\C \times \CP^1$.
\end{thm}

\begin{proof}
By taking the universal cover, we may assume that $M$ is simply connected. From Lemma \ref{lem:405}, we know that $(M,g(-\tau))$ has nonnegative curvature operator. From \cite[Theorem $5.3$]{NT}, we conclude that $(M,g(0))$ is biholomorphic-isometric to $N \times L$ where $N$ is a compact Hermitian symmetric space and $L$ is diffeomorphic to $\R^{2k}$ where $k=\dim_{\C} L$. If $k=1$, then $(M,g(t))$ is biholomorphic-isometric to the product of a real 2-dimensional $\kappa$-solutions $L$ and $\CP^1$. From the result in \cite{P}, it implies that $M$ is biholomorphic-isometric to $\C \times \CP^1$ that verifies the statement. 

If $k=2$, then $M=L$ and in particular $M$ is diffeomorphic to $\R^4$. From Proposition \ref{prop:glob}, there exists a global $S^2$-fibration over a noncompact Riemann surface $\mathcal S$. Applying the long exact sequence of homotopy groups on the fibration $S^2 \to \R^4 \to \mathcal S$, we have
$$
\pi_1(S^2)\to (\pi_1(\R^4)=0) \to \pi_1(\mathcal S) \to (\pi_0(S^2)=0),
$$
which shows $\pi_1(\mathcal S)=0$, i.e., $\mathcal S$ is simply connected. By the classification of noncompact surfaces, $\mathcal S$ is diffeomorphic to $\R^2$. In particular, the base space of this fibration is contractible. Therefore, this fibration is trivial so the total space has to be diffeomorphic to $S^2 \times \R^2$, which is a contradiction. From Lemma \ref{lem:quo}, the proof is complete.
\end{proof}

Now the proof of Theorem \ref{thm:main2} is immediate.

\emph{Proof of Theorem \ref{thm:main2}}: Theorem \ref{thm:main2} follows from Theorem \ref{thm:cpt} and Theorem \ref{thm:kappa3}. 

\begin{rem}
With more effort, we can classify all $\kappa$-solutions to the \ka Ricci flow for any dimensions. Moreover, we can show any complete, $\kappa$-noncollapsed, ancient solutions to the \ka Ricci flow with nonnegative bisectional curvature must be a $\kappa$-solution. Those results will appear in a separate paper.
\end{rem}

\section{Further discussion}
In this section, we propose some conjectures.

Based on the result of 4-dimensional Ricci shrinkers with weakly PIC \cite{LNW18}, we propose the following conjecture, see also \cite{BE10}.

\begin{conj}\label{conj:1}
Let $(M^n,g,f)$ be a Ricci shrinker with weakly \emph{PIC}. Then $(M,g)$ is isometric to a quotient of $N^k \times \R^{n-k}$, where $N$ is a compact symmetric space.
\end{conj}

In general, we have the following conjecture, which directly implies Conjecture \ref{conj:1}.

\begin{conj}\label{conj:2}
Let  $(M^n,g(t))_{t \in (-\infty,0]}$ be a complete ancient solution to the Ricci flow with weakly \emph{PIC}. Then it has weakly \emph{PIC}$_2$.
\end{conj}

\vskip10pt

Jae Ho Cho, Department of Mathematics, Stony Brook University, Stony Brook, NY 11794, USA; jaeho.cho@stonybrook.edu.
\vskip10pt
Yu Li, Department of Mathematics, Stony Brook University, Stony Brook, NY 11794, USA; yu.li.4@stonybrook.edu.\\

\end{document}